\documentclass{amsart} 
\usepackage{latexsym}
\usepackage{amsfonts}
\usepackage{amssymb}
\usepackage{amsmath}
\usepackage{amscd}
\usepackage{graphicx}
\usepackage{eucal}
\usepackage{amsthm}
\usepackage[all]{xy}
\usepackage{pdfsync}
\usepackage{xspace}
\usepackage{hyperref}
\hypersetup{unicode=true,pdfusetitle,pdfpagelabels,
 bookmarks=true,bookmarksnumbered=false,
 breaklinks=false,
 backref=false,
 colorlinks=true,
 linkcolor=blue,
 citecolor=blue,
 urlcolor=blue,
 final
}
\usepackage[capitalise]{cleveref}

\newcommand{\newrefformat}[2]{}
\newcommand{\cn}{\colon}

\crefname{lem}{Lemma}{Lemmas}
\crefname{thm}{Theorem}{Theorems}
\crefname{Defi}{Definition}{Definitions}
\crefname{nota}{Notation}{Notations}
\crefname{construction}{Construction}{Constructions}
\crefname{prop}{Proposition}{Propositions}
\crefname{rem}{Remark}{Remarks}
\crefname{cor}{Corollary}{Corollaries}
\crefname{scholium}{Scholium}{Scholia}
\crefname{figure}{Figure}{Figures}
\crefname{equation}{Equation}{Equations}
\crefname{eq}{Equation}{Equations}
\crefname{eqn}{Equation}{Equations}

\newenvironment{eqn}{\begin{equation}}{\end{equation}}
\newenvironment{eq*}{\begin{equation*}}{\end{equation*}}
\newenvironment{eqn*}{\begin{equation*}}{\end{equation*}}

\newcommand{\mb}{\mathbf}

\renewcommand{\SS}{\mathcal{S}}
\newcommand{\colim}{\textrm{colim }}

\def\srarrow{\relbar\joinrel\mapstochar\joinrel\rightarrow}
\newcommand{\id}{\ensuremath{\textnormal{id}}\xspace}

\begin{document}
\numberwithin{equation}{section}
\newtheorem{thm}[equation]{Theorem}
\newtheorem{prop}[equation]{Proposition}
\newtheorem{lem}[equation]{Lemma}
\newtheorem{cor}[equation]{Corollary}

\newtheoremstyle{example}{\topsep}{\topsep}%
     {}
     {}
     {\bfseries}
     {.}
     {2pt}
     {\thmname{#1}\thmnumber{ #2}\thmnote{ #3}}

   \theoremstyle{example}
   \newtheorem{nota}[equation]{Notation}
   \newtheorem{example}[equation]{Example}
   \newtheorem{Defi}[equation]{Definition}
   \newtheorem{rem}[equation]{Remark}
	\newtheorem{comment}[equation]{Comment}

   \title{Operads, tensor products, and the categorical Borel construction}

\author{Nick Gurski}
\address{
School of Mathematics and Statistics,
University of Sheffield,
Sheffield, UK, S3 7RH
}
\email{nick.gurski@sheffield.ac.uk}

\begin{abstract}
We show that every action operad gives rise to a notion of monoidal category via the categorical version of the Borel construction, embedding action operads into the category of 2-monads on $\mathbf{Cat}$.  We characterize those 2-monads in the image of this embedding, and as an example show that the theory of coboundary categories corresponds precisely to the operad of $n$-fruit cactus groups.  We finally define $\mathbf{\Lambda}$-multicategories for an action operad $\mathbf{\Lambda}$, and show that they arise as monads in a Kleisli bicategory.
\end{abstract}

\maketitle

\tableofcontents
	
\section*{Introduction}

Categories of interest are often monoidal: sets, topological spaces, and vector spaces are all symmetric monoidal, while the category of finite ordinals (under ordinal sum) is merely monoidal.  But other categories have more exotic monoidal structures.  The first such type of structure discovered was that of a braided monoidal category.  These arise in categories whose morphisms have a geometric flavor like cobordisms embedded in some ambient space \cite{js}, in  categories produced from double loop spaces \cite{fie-br}, and categories of representations over objects like quasitriangular (or braided) bialgebras \cite{street-quantum} .  Another such exotic monoidal structure is that of a coboundary category, arising in examples from the representation theory of quantum groups \cite{drin-quasihopf}.

Going back to the original work of May on iterated loop spaces \cite{maygeom}, operads were defined in both symmetric and nonsymmetric varieties.  But Fiedorowicz's work on double loop spaces \cite{fie-br} showed that there was utility in considering another kind of operad, this time with braid group actions instead of symmetric group actions.  There is a clear parallel between these definitions of different types of operads and the definitions of different kinds of monoidal category, with each given by some general schema in which varying an $\mathbb{N}$-indexed collection of groups produced the types of operads or monoidal categories seen in nature.  Building on the work in \cite{cg}, the goal of this paper is to show that this parallel can be upgraded from an intuition to precise mathematics using the notion of action operad.

An action operad $\mb{\Lambda}$ is an operad which incorporates all of the essential features of the operad of symmetric groups.  Thus $\Lambda(n)$ is no longer just a set, but instead also has a group structure together with a map $\pi_{n}:\Lambda(n) \to \Sigma_{n}$.  Operadic composition then satisfies an additional equivariance condition using the maps $\pi_n$ and the group structures.  Each action operad $\mb{\Lambda}$ produces a notion of $\mb{\Lambda}$-operad which encodes equivariance conditions using both the groups $\Lambda(n)$ and the maps $\pi_n$.  Examples include the symmetric groups, the terminal groups (giving nonsymmetric operads), the braid groups (giving braided operads), and the $n$-fruit cactus groups \cite{hk-cobound} (giving a new notion of operad one might call cactus operads).  Using a formula resembling the classical Borel construction for spaces with a group action, we can produce from any action operad $\mb{\Lambda}$ a notion of $\mb{\Lambda}$-monoidal category, in which the group $\Lambda(n)$ acts naturally on $n$-fold tensor powers of any object.  Thus the categorical Borel construction embeds action operads into a category of monads on $\mb{Cat}$, and we characterize the image of this embedding as those monads describing monoidal structures of a precise kind.

The paper is organized into the following sections.  Section 1 reviews the definition of an action operad, and defines the categorical Borel construction on them.  The key result, which reappears in proofs throughout the paper, is \cref{thm:charAOp}, characterizing action operads in terms of two new operations mimicking the block sum of permutations and the operation which takes a permutation of $n$ letters and produces a new permutation on $k_1 + k_2 + \cdots + k_n$ letters by permuting the blocks of $k_i$ letters.  In Section 2, we use this characterization and Kelly's theory of clubs \cite{kelly_club1, kelly_club0, kelly_club2} to embed action operads into monads on $\mb{Cat}$ and determine the essential image of this embedding.  Section 3 gives a construction of the free action operad from a suitable collection of data, and relates this to how clubs can be described using generators and relations.  The results of Sections 2 and 3 show that the definitions of symmetric monoidal category or coboundary category, for example, correspond to the action operad constructed from the corresponding free symmetric monoidal or coboundary category on one object; these and other examples appear in detail in Section 4.  Section 5 then extends the definition of $\mb{\Lambda}$-operad to that of $\mb{\Lambda}$-multicategory and shows that these arise abstractly via a Kleisli construction.

The author would like to thank Alex Corner and Ed Prior for conversations contributing to this research.

This research was supported by EPSRC 134023.

\section{The Borel construction for action operads}

The clasical Borel construction is a functor from $G$-spaces to spaces, sending a $G$-space $X$ to $EG \times_{G} X$.  Our goal in this section is to use the formal description of the Borel construction to construct some special operads in $\mb{Cat}$.   We start by reviewing the analogues of the functors $E, B:\mb{Grp} \rightarrow \mb{Top}$ now taking values in $\mb{Cat}$.

\begin{Defi}\label{Defi:e_b}
\begin{enumerate}
\item Let $X$ be a set.  We define the \textit{translation category} $EX$ to have objects the elements of $X$ and morphisms consisting of a unique isomorphism between any two objects.
\item Let $G$ be a group.  The category $BG$ has a single object $*$, and hom-set $BG(*,*) = G$ with composition and identity given by multiplication and the unit element in the group, respectively.
\end{enumerate}
\end{Defi}

\begin{Defi}
A functor $F:X \to Y$ is an \emph{isofibration} if given $x \in X$ and an isomorphism $f: y \cong F(x)$ in $Y$, there is an isomorphism $g$ in $X$ such that $F(g) = f$.
\end{Defi}

\begin{prop}
There is a natural transformation $p:EU \Rightarrow B$, where $U$ is the underlying set of a group, which is pointwise an isofibration.  Applying the classifying space functor to the component $p_{G}$ gives a universal principal $G$-bundle.
\end{prop}
\begin{proof}
Given a group $G$, $p_{G}:EUG \rightarrow BG$ sends every object of $EUG$ to the unique object of $BG$.  The unique isomorphism $g \rightarrow  h$ in $EUG$ is mapped to $hg^{-1}:* \rightarrow *$.  It is easy to directly check that this is an isofibration, as well as to see that the classifying spaces of $EUG$ and $BG$ are the spaces classically known as $EG,BG$, with $|p_{G}|$ being the standard universal principal $G$-bundle.
\end{proof}

We will also need the functors $E, B$ defined for more than just a single set or group, in particular for the sets or groups which make up an operad and are indexed by the natural numbers.

\begin{nota}\label{nota:e_b}
Let $S$ be a set which we view as a discrete category.
\begin{enumerate}
\item For any functor $F:S \rightarrow \mb{Sets}$, let $EF$ denote the composite $E \circ F:S \rightarrow \mb{Cat}$; we often view $F$ as an indexed set $\{ F(s) \}$, in which case $EF$ is the indexed category $\{ EF(s) \}$.
\item For any functor $F:S \rightarrow \mb{Grp}$, let $BF$ denote the composite $B \circ F:S \rightarrow \mb{Cat}$; we often view $F$ as an indexed group $\{ F(s) \}$, in which case $BF$ is the indexed category $\{ BF(s) \}$.
\end{enumerate}
\end{nota}

The following lemma is a straightforward verification.

\begin{lem}\label{symmoncor}
The functor $E:\mb{Sets} \rightarrow \mb{Cat}$ is right adjoint to the set of objects functor.  Therefore $E$ preserves all limits, and in particular is a symmetric monoidal functor when both categories are equipped with their cartesian monoidal structures.
\end{lem}

We now recall the definitions of an action operad $\mb{\Lambda}$ from \cite{cg}.

\begin{Defi}\label{Defi:aop}
An \textit{action operad} $\mb{\Lambda}$ consists of
\begin{itemize}
\item an operad $\Lambda = \{ \Lambda(n) \}$ in the category of sets such that each $\Lambda(n)$ is equipped with the structure of a group and
\item a map $\pi:\Lambda \rightarrow \Sigma$ which is simultaneously a map of operads and a group homomorphism $\pi_{n}:\Lambda(n) \rightarrow \Sigma_{n}$ for each $n$
\end{itemize}
such that one additional axiom holds.  Write
\[
\mu: \Lambda(n) \times \Lambda(k_{1}) \times \cdots \times \Lambda(k_{n}) \rightarrow \Lambda(k_{1} + \cdots + k_{n})
\]
for the multiplication in the operad $\Lambda$.  Let $(g; f_1, \ldots f_n)$ be an element of the product $\Lambda(n) \times \Lambda(k_{\pi(g')^{-1}(1)}) \times \cdots \times \Lambda(k_{\pi(g')^{-1}(1)})$ and $(g'; f_1', \ldots, f_n')$ be an element of the product $\Lambda(n) \times \Lambda(k_{1}) \times \cdots \times \Lambda(k_{n})$.  We require that
\begin{eqn}\label{eqn:ao_axiom}
\mu(g; f_1, \ldots f_n)  \mu(g'; f_1', \ldots, f_n') = \mu (gg'; f_{\pi(g')(1)}f_{1}', \ldots, f_{\pi(g')(n)}f_{n}')
\end{eqn}in the group $\Lambda(k_{1} + \cdots + k_{n})$.
\end{Defi}

\begin{example}
\begin{enumerate}
\item The terminal operad $T$ in the category of sets has a unique action operad structure, $\mathbf{T}$.  Since $T(n)$ is a singleton for each $n$, the group structure is unique as is the map $\pi$.  The single action operad axiom is then automatic as both sides of \cref{eqn:ao_axiom} are the unique element which happens to be the identity.  This is the initial object in the category of action operads.
\item The symmetric operad $\Sigma$ has a canonical action operad structure.  It is given by taking $\pi$ to be the identity map, and this action operad will be denoted $\mathbf{\Sigma}$.  This is the terminal object in the category of action operads.
\item The operad of braid groups also has an obvious action operad structure in which the map $\pi$ is the group homomorphism sending a braid to its underlying permutation.  We will denote this action operad by $\mathbf{Br}$.  The operad of ribbon braids has an action operad structure, essentially using the same map $\pi$, and was studied by Wahl in \cite{wahl-thesis}.
\item The operad of $n$-fruit cactus groups defined by Henriques and Kamnitzer in \cite{hk-cobound} has an action operad structure that we will discuss in Section 4.
\end{enumerate}
\end{example}

Note that the operad of symmetric groups $\Sigma$ has its action operad structure determined by two auxilliary operations.  The first is the block sum of permutations which we denote by
\[
\beta: \Sigma_{k_{1}} \times \cdots \times \Sigma_{k_{n}} \rightarrow \Sigma_{\underline{k}},
\]
where $\underline{k} = \sum k_{i}$.  The second is a kind of diagonal map which is defined for any natural number $n$ together with natural numbers $k_{1}, \ldots, k_{n}$.  Then
\[
\delta = \delta_{n; k_{1}, \ldots, k_{n}}:\Sigma_{n} \rightarrow \Sigma_{\underline{k}},
\]
is defined on $\sigma \in \Sigma_{n}$ by permuting the elements $1, 2, \ldots, k_{1}$ together in a block according to the action of $\sigma \in \Sigma_{n}$ on $1$, then $k_{1}+1, \ldots, k_{1}+k_{2}$ together in a block according to the action of $\sigma$ on $2$, and so on.  The first of these, $\beta$, is a group homomorphism, while $\delta$ is a sort of twisted homomorphism, and taken together they define operadic multiplication in $\Sigma$.  We now use these ideas to give the following algebraic characterization of action operads.

\begin{nota}
We will denote our identity elements in groups generically as $e$. If $\{ G_{i} \}_{i \in I}$ are groups indexed by a set $I$, then $e_{i}$ is the identity element in $G_{i}$.
\end{nota}

\begin{thm}\label{thm:charAOp}
An action operad $\mb{\Lambda}$ determines, and is determined by, the following:
\begin{itemize}
\item groups $\Lambda(n)$ together with group homomorphisms $\pi_{n}:\Lambda(n) \rightarrow \Sigma_{n}$,
\item a group homomorphism
\[
\begin{array}{rcl}
\Lambda(k_{1}) \times \cdots \times \Lambda(k_{n}) & \stackrel{\beta}{\longrightarrow} & \Lambda(k_{1} + \cdots + k_{n}),
\end{array}
\]
for each $k_{1}, \ldots, k_{n}$, and
\item a function of sets
\[
\begin{array}{rcl}
\Lambda(n) & \stackrel{\delta_{n; k_{1}, \ldots, k_{n}}}{\longrightarrow} & \Lambda(k_{1} + \cdots + k_{n})
\end{array}
\]
for each $n, k_{1}, \ldots, k_{n}$,
\end{itemize}
subject to the following axioms.
\begin{enumerate}
\item\label{eq1} The homomorphisms $\beta$ are natural with respect to the maps $\pi_{n}$, where $\underline{k} = k_{1} + \cdots + k_{n}$.
\[
\xy
(0,0)*+{\Lambda(k_{1}) \times \cdots \times \Lambda(k_{n}) } ="00";
(0,-15)*+{\Sigma_{k_{1}} \times \cdots \times \Sigma_{k_{n}}  } ="01";
(40,0)*+{\Lambda(\underline{k}) } ="20";
(40,-15)*+{\Sigma_{\underline{k}} } ="21";
{\ar^{\beta} "00" ; "20"};
{\ar^{\pi} "20" ; "21"};
{\ar_{\pi_1 \times \cdots \times \pi_n} "00" ; "01"};
{\ar_{\beta} "01" ; "21"};
\endxy
\]
\item\label{eq2} The homomorphism $\beta:\Lambda(k) \to \Lambda(k)$ is the identity.
\item\label{eq3} The homomorphisms $\beta$ are associative in the sense that
\[
\beta(k_{11}, \ldots, k_{1j_{1}}, k_{21}, \ldots, k_{2j_{2}}, \ldots, k_{nj_{n}}) = \beta\big( \beta(k_{11}, \ldots, k_{1j_{1}}), \ldots, \beta(k_{n1}, \ldots, k_{nj_{n}}) \big)
\]
holds.
\item\label{eq4} The functions $\delta_{n; k_{1}, \ldots, k_{n}}$ are natural with respect to the maps $\pi_{n}$.
\[
\xy
(0,0)*+{\Lambda(n)} ="00";
(40,0)*+{\Lambda(k_{1} + \cdots + k_{n}) } ="20";
(0,-15)*+{\Sigma_{n}  } ="01";
(40,-15)*+{\Sigma_{k_{1} + \cdots + k_{n}} } ="21";
{\ar^{\delta} "00" ; "20"};
{\ar^{\pi} "20" ; "21"};
{\ar_{\pi} "00" ; "01"};
{\ar_{\delta} "01" ; "21"};
\endxy
\]
\item\label{eq5} The functions $\delta_{n; 1, \ldots, 1}, \, \delta_{1;n} : \Lambda(n) \to \Lambda(n)$ are the identity.
\item\label{eq6} The equation $\delta_{n; k_{i}}(g) \delta_{n; j_{i}}(h) = \delta_{n; j_{i}}(gh)$ holds when
\[
k_{i} = j_{\pi(h)^{-1}(i)}.
\]
\item\label{eq7} The functions $\delta$ are associative in the sense that
\[
\delta_{m_1 + \cdots + m_n; p_{11}, \ldots, p_{1m_{1}}, p_{21}, \ldots, p_{nm_{m}}}\big( \delta_{n; m_{1}, \ldots, m_{n}}(f) \big) = \delta_{n; P_{1}, \ldots, P_{n}}(f)
\]
where $P_{i} = p_{i1} + \cdots + p_{im_{i}}$.
\item\label{eq8} $\delta(g) \beta(h_{1}, \ldots, h_{n}) = \beta(h_{\pi(g)^{-1}(1)}, \ldots,  h_{\pi(g)^{-1}(n)}) \delta(g)$, where $h_{i} \in \Lambda(k_{i})$ and $\delta:\Lambda(n) \rightarrow \Lambda(k_{1} + \cdots + k_{n})$.
\item\label{eq9} The equation
\[
\beta(\delta_{1}(g_{1}), \ldots, \delta_{n}(g_{n})) = \delta_{c}(\beta(g_{1}, \ldots, g_{n}))
\]
holds, where $\delta_{i}(g_{i})$ is shorthand for $\delta_{k_{i}; m_{i1}, \ldots, m_{ik_{i}}}(g_{i})$ and $\delta_{c}$ is shorthand for
\[
\delta_{k_{1}+\cdots + k_{n}; m_{11}, m_{12}, \ldots, m_{1k_{1}}, m_{21}, \ldots, m_{nk_{n}}}.
\]
\end{enumerate}
\end{thm}

\begin{proof}
For an action operad $\mb{\Lambda}$, define
\[
\begin{array}{rcl}
\beta(h_{1}, \ldots, h_{n}) &=& \mu(e; h_{1}, \ldots, h_{n}) \\
\delta(g) &=& \mu(g; e_{1}, \ldots, e_{n}).
\end{array}
\]
Since $\pi:\Lambda \rightarrow \Sigma$ is an operad map, axioms \eqref{eq1} and \eqref{eq4} hold.  Since $\Lambda$ is an operad of sets, axioms \eqref{eq2} and \eqref{eq5} follow from the operad unit axioms, and axioms \eqref{eq3}, \eqref{eq7}, and \eqref{eq9} follow from the operad associativity axiom.  Axioms \eqref{eq6} and \eqref{eq8} are special cases of the additional action operad axiom, as is the fact that $\beta$ is a group homomorphism.

Given the data above, we need only define the operad multiplication, verify the operad unit and multiplication axioms,  and finally check the action operad axiom.  Multiplication is given by
\[
\mu(g; h_{1}, \ldots, h_{n}) = \delta_{n; k_{1}, \ldots, k_{n}}(g) \beta(h_{1}, \ldots, h_{n})
\]
where $h_{i} \in \Lambda(k_{i})$.  The unit is $e \in \Lambda(1)$.

We now verify the operad unit axioms.
\[
\begin{array}{rcl}
\mu(e; g) & = & \delta(e)\beta(g) \\
& = & e \cdot g \\
& = & g
\end{array}
\]
\[
\begin{array}{rcl}
\mu(h; e, \ldots, e) & = & \delta(h)\beta(e, \ldots, e) \\
& = & h \cdot e \\
& = & h
\end{array}
\]
These follow from axioms \eqref{eq2} and \eqref{eq5}, together with the fact that $\beta$ is a group homomorphism.

For the operad associativity axiom, let
\begin{itemize}
\item $f \in \Lambda(m),$
\item $g_{i} \in \Lambda(n_{i})$ for $i=1, \ldots, m$, and
\item $h_{ij} \in \Lambda(p_{i,j})$ for $i=1, \ldots, m$ and $j=1, \ldots, n_{i}$.
\end{itemize}
Further, let $P_{i} = p_{i1} + \cdots + p_{in_{i}}$ and $\underline{h_{i\bullet}}$ denote the list $h_{i1}, h_{i2}, \ldots, h_{in_{i}}$.  We must then show that
\[
\mu\big( f; \mu(g_{1}; \underline{h_{1\bullet}}), \ldots, \mu(g_{m}; \underline{h_{m\bullet}}) \big) = \mu\big( \mu (f; g_{1}, \ldots, g_{m}); h_{11}, \ldots, h_{1n_{1}}, h_{21}, \ldots, h_{mn_{m}} \big).
\]
By definition, the left side of this equation is
\[
\delta_{m; P_{1}, \ldots, P_{m}}(f) \beta\big( \mu(g_{1}; \underline{h_{1\bullet}}), \ldots, \mu(g_{m}; \underline{h_{m\bullet}}) \big),
\]
and
\[
\mu(g_{i}; \underline{h_{i\bullet}}) = \delta_{n_{i}; p_{i1}, \ldots, p_{in_{i}}}(g_{i})\beta(h_{i1}, \ldots, h_{in_{i}}).
\]
Since $\beta$ is a group homomorphism, we can then rewrite the left side as
\[
\delta(f)\beta\big(\delta(g_{1}), \ldots, \delta(g_{m})\big)\beta\big(\beta(h_{1\bullet}), \ldots, \beta(h_{m\bullet})\big)
\]
where we have suppressed the subscripts on the $\delta$'s.  By axiom \eqref{eq3} , we have
\[
\beta\big(\beta(h_{1\bullet}), \ldots, \beta(h_{m\bullet})\big) = \beta(h_{11}, \ldots, h_{1n_{1}}, h_{21}, \ldots, h_{mn_{m}}).
\]
Further, axiom \eqref{eq9} above shows that
\[
\beta\big(\delta(g_{1}), \ldots, \delta(g_{m})\big) = \delta\big(\beta(g_{1}, \ldots, g_{m})\big).
\]
Thus we have shown that the left side of the operad associativity axiom is equal to
\[
\delta(f)\delta\big(\beta(g_{1}, \ldots, g_{m})\big)\beta(h_{11}, \ldots, h_{1n_{1}}, h_{21}, \ldots, h_{mn_{m}}).
\]
Now the right side is
\[
\mu\big( \mu (f; g_{1}, \ldots, g_{m}); h_{11}, \ldots, h_{1n_{1}}, h_{21}, \ldots, h_{mn_{m}} \big)
\]
which is by definition
\[
\delta\big(\mu (f; g_{1}, \ldots, g_{m})\big)\beta(h_{11}, \ldots, h_{1n_{1}}, h_{21}, \ldots, h_{mn_{m}} ).
\]
Thus verifying the operad associativity axiom reduces to showing
\begin{eqn}\label{eqn:opass}
\delta(f)\delta\big(\beta(g_{1}, \ldots, g_{m})\big) = \delta\big(\mu (f; g_{1}, \ldots, g_{m})\big).
\end{eqn}By the definition of $\mu$, we have
\[
\delta(\mu (f; g_{1}, \ldots, g_{m})) = \delta\big(\delta(f)\beta(g_{1}, \ldots, g_{m}) \big)
\]
which is itself equal to
\begin{eqn}\label{eqn:opass2}
\delta\big(\delta(f)\big) \delta\big(\beta(g_{1}, g_{m})\big)
\end{eqn}by axiom \eqref{eq6} above.  Now the $\delta(f)$ on the left side of \cref{eqn:opass} uses $\delta_{n; P_{1}, \ldots, P_{n}}$, while the $\delta(\delta(f))$ in \cref{eqn:opass2} is actually
\[
\delta_{m_1 + \cdots + m_{n}; q_{ij}}(\delta_{n; m_{1}, \ldots, m_{n}} (f))
\]
where the $q_{ij}$ are defined, by axiom \eqref{eq6}, to be given by
\[
q_{ij} = p_{i,\pi(g_{i})^{-1}(j)}
\]
using the compatibility of $\beta$ and $\pi$ in axiom \eqref{eq1}.  By axiom \eqref{eq7}, this is composite of $\delta$'s  is then $\delta_{n; Q_{1}, \ldots, Q_{n}}$ where $Q_{i} = q_{i1} + \cdots + q_{im_{i}}$.  But by the definition of the $q_{ij}$, we immediately see that $Q_{i} = P_{i}$, so the $\delta(f)$ in \cref{eqn:opass} is equal to the $\delta(\delta(f))$ appearing in \cref{eqn:opass2}, concluding the proof of the operad associativity axiom.

The action operad axiom is now the calculation below, and uses axioms \eqref{eq4} and \eqref{eq8}.
\begin{small}
\[
\begin{array}{rcl}
\mu(g; h_{1}, \ldots, h_{n})\mu(g'; h_{1}', \ldots, h_{n}') & = & \delta(g) \beta(h_{1}, \ldots, h_{n}) \delta(g') \beta(h_{1}', \ldots, h_{n}') \\
& = & \delta(g) \delta(g') \beta(h_{\pi(g')(1)}, \ldots, h_{\pi(g')(n)})  \beta(h_{1}', \ldots, h_{n}') \\
& = & \delta(gg') \beta(h_{\pi(g')(1)}h_{1}', \ldots, h_{\pi(g')(n)}h_{n}') \\
& = & \mu(gg'; h_{\pi(g')(1)}h_{1}', \ldots, h_{\pi(g')(n)}h_{n}')
\end{array}
\]
\end{small}
\end{proof}

\begin{Defi}\label{Defi:cat_aop}
The category $\mb{AOp}$ of action operads has
\begin{itemize}
\item objects which are action operads $\mb{\Lambda}$, and
\item morphisms $\mb{\Lambda} \rightarrow \mb{\Lambda'}$ which are those operad maps $f:\Lambda \rightarrow \Lambda'$ such that each $f_{n}:\Lambda(n) \rightarrow \Lambda'(n)$ is a group homomorphism and $\pi_{n}^{\Lambda} = \pi_{n}^{\Lambda'} \circ f_{n}$.
\end{itemize}
\end{Defi}

\begin{Defi}\label{Defi:lamop}
Let $\mb{\Lambda}$ be an action operad.  A \textit{$\mb{\Lambda}$-operad} $P$ (in $\mb{Sets}$) consists of
\begin{itemize}
\item an operad $P$ in $\mb{Sets}$ and
\item for each $n$, an action $P(n) \times \Lambda(n) \rightarrow P(n)$ of $\Lambda(n)$ on $P(n)$
\end{itemize}
such that the following two equivariance axioms hold for $P$.
\[
\begin{array}{rcl}
\mu^{P}(x; y_{1} \cdot g_{1}, \ldots, y_{n} \cdot g_{n}) & = & \mu^{P}(x; y_{1}, \ldots, y_{n}) \cdot \mu^{\Lambda}(e; g_{1}, \ldots, g_{n})  \\
\mu^{P}(x \cdot g; y_{1}, \ldots, y_{n}) &  = & \mu^{P}(x; y_{\pi(g)^{-1}(1)}, \ldots, y_{\pi(g)^{-1}(n)}) \cdot \mu^{\Lambda}(g; e_{1}, \ldots, e_{n})
\end{array}
\]
\end{Defi}

We are additionally interested in $\mb{\Lambda}$-operads in $\mb{Cat}$ (or other cocomplete symmetric monoidal categories in which tensoring with a fixed object preserves colimits).  While the definition above gives the correct notion of a $\mb{\Lambda}$-operad in $\mb{Cat}$ if we interpret the two equivariance axioms to hold for both objects and morphisms, it is useful to give a purely diagrammatic expression of these axioms.  In the diagrams below, expressions of the form $G \times C$ for a group $G$ and category $C$ mean that the group $G$ is to be treated as a discrete category.  This follows that standard method of how one expresses group actions in categories other than $\mb{Sets}$ using a copower.  Thus the diagrams below are the two equivariance axioms given in \cref{Defi:lamop} expressed diagrammatically.

\[
\xy
(0,0)*+{\scriptstyle P(n) \times P(k_{1}) \times \cdots \times P(k_{n}) \times \Lambda(k_{1}) \times \cdots \times \Lambda(k_{n}) } ="00";
(0,-15)*+{\scriptstyle P(\underline{k}) \times \Lambda(\underline{k}) } ="01";
(70,0)*+{\scriptstyle P(n) \times P(k_{1}) \times \Lambda(k_{1}) \times \cdots \times P(k_{n}) \times  \Lambda(k_{n}) } ="20";
(70,-15)*+{\scriptstyle P(n) \times P(k_{1}) \times \cdots \times P(k_{n}) } ="21";
(35, -25)*+{\scriptstyle P(\underline{k}) } ="12";
{\ar^{\cong} "00" ; "20"};
{\ar^{\alpha_{k_1} \times \cdots \times \alpha_{k_n}} "20" ; "21"};
{\ar^{\mu^P} "21" ; "12"};
{\ar_{\mu^P \times \mu^\Lambda(e;-)} "00" ; "01"};
{\ar_{\alpha_{\underline{k}}} "01" ; "12"};
\endxy
\]
\[
\xy
(0,0)*+{\scriptstyle P(n) \times \Lambda(n) \times P(k_{1}) \times \cdots \times P(k_{n}) } ="00";
(0,-10)*+{\scriptstyle P(n) \times \Lambda(n) \times \Lambda(n) \times P(k_{1}) \times \cdots \times P(k_{n}) } ="01";
(0,-20)*+{\scriptstyle P(n) \times \Lambda(n) \times P(k_{1}) \times \cdots \times P(k_{n}) \times \Lambda(n) } ="02";
(0,-30)*+{\scriptstyle P(n) \times \Sigma_{n} \times P(k_{1}) \times \cdots \times P(k_{n}) \times \Lambda(n) } ="03";
(55,-30)*+{\scriptstyle P(\underline{k}) \times \Lambda(\underline{k}) } ="13";
(70,0)*+{\scriptstyle P(n) \times P(k_{1}) \times \cdots \times P(k_{n}) } ="20";
(70,-18)*+{\scriptstyle P(\underline{k}) } ="21";
{\ar_{1 \times \Delta \times 1} "00" ; "01"};
{\ar^{\cong} "01" ; "02"};
{\ar_{1 \times \pi_{n} \times 1} "02" ; "03"};
{\ar^{} "03" ; "13"};
(35,-33)*{\scriptstyle \tilde{\mu}^P \times \mu^\Lambda(-;\underline{e})};
{\ar_{\alpha_{\underline{k}}} "13" ; "21"};
{\ar^{\alpha_{n} \times 1} "00" ; "20"};
{\ar^{\mu^P} "20" ; "21"};
\endxy
\]

\begin{Defi}\label{Defi:actop_to_cat}
Let $\mb{\Lambda}$ be an action operad.  Then $B\mb{\Lambda}$ (see \cref{nota:e_b}) is the category with objects the natural numbers and
\[
B\mb{\Lambda}(m,n) = \left\{ \begin{array}{lc}
\Lambda(n), & m = n \\
\varnothing, & m \neq n,
\end{array} \right.
\]
where composition is given by group multiplication and the identity morphism is the unit element $e_n \in \Lambda(n)$.
\end{Defi}

\begin{thm}\label{preserveGop}
Let $M,N$ be cocomplete symmetric monoidal categories in which the tensor product preserves colimits in each variable, and let $F:M \rightarrow N$ be a symmetric lax monoidal functor with unit constraint $\varphi_{0}$ and tensor constraint $\varphi_{2}$.  Let $\mb{\Lambda}$ be an action operad, and $P$ a $\mb{\Lambda}$-operad in $M$.  Then $FP = \{ F(P(n)) \}$ has a canonical $\mb{\Lambda}$-operad structure, giving a functor
\[
\mb{\Lambda}\mbox{-}Op(M) \rightarrow \mb{\Lambda}\mbox{-}Op(N)
\]
from the category of $\mb{\Lambda}$-operads in $M$ to the category of $\mb{\Lambda}$-operads in $N$.
\end{thm}
\begin{proof}
The category $\mb{\Lambda}$-operads in $M$ is the category of monoids for the composition product $\circ_{M}$ on $[B\mb{\Lambda}^{\textrm{op}}, M]$ constructed in \cite{cg}.  Composition with $F$ gives a functor
\[
F_{*}: [B\mb{\Lambda}^{\textrm{op}}, M] \rightarrow [B\mb{\Lambda}^{\textrm{op}}, N],
\]
and to show that it gives a functor between the categories of monoids we need only prove that $F_{*}$ is lax monoidal with respect to $\circ_{M}$ and $\circ_{N}$.  In other words, we must construct natural transformations with components $FX \circ_{N} FY \rightarrow F(X \circ_{M} Y)$ and $I_{Op(N)} \rightarrow F(I_{Op(M)})$ and then verify the lax monoidal functor axioms.

We first remind the reader about copowers in cocomplete categories.  For an object $X$ and set $S$, the copower $X \odot S$ is the coproduct $\coprod_{s \in S} X$.  We have natural isomorphisms $(X \odot S) \odot T \cong X \odot (S \times T)$ and $X \odot 1 \cong X$, and using these we can define an action of a group $G$ on an object $X$ using a map $X \odot G \rightarrow X$.  Any functor $F$ between categories with coproducts is lax monoidal with respect to those coproducts:  the natural map $FA \coprod FB \rightarrow F(A \coprod B)$ is just the map induced by the universal property of the coproduct using $F$ applied to the coproduct inclusions $A \hookrightarrow A \coprod B, B \hookrightarrow A \coprod B$.  In particular, for any functor $F$ we get an induced map $FX \odot S \rightarrow F(X \odot S)$.

The unit object in $[B\mb{\Lambda}^{\textrm{op}}, M]$ for $\circ_{M}$ is the copower $I_{M} \odot B\mb{\Lambda}(-,1)$.  Thus the unit constraint for $F_{*}$ is the composite
\[
I_{N} \odot B\mb{\Lambda}(-,1) \stackrel{\varphi_{0} \odot 1}{\longrightarrow} FI_{M} \odot B\mb{\Lambda}(-,1) \rightarrow F(I_{M} \odot B\mb{\Lambda}(-,1) ).
\]

For the tensor constraint, we will require a map
\[
t:(FY)^{\star n}(k) \rightarrow F\big(Y^{\star n}(k)\big)
\]
 where $\star$ is the Day convolution product; having constructed one, the tensor constraint is then the composite
\[
\begin{array}{rcl}
(FX \circ FY)(k) & \cong & \int^{n} FX(n) \otimes (FY)^{\star n}(k) \\
& \stackrel{ \int 1 \otimes t}{\longrightarrow}  & \int^{n} FX(n) \otimes F(Y^{\star n}(k)) \\
& \stackrel{\int \varphi_{2}}{\longrightarrow}  & \int^{n} F(X(n) \otimes Y^{\star n}(k)) \\
& \longrightarrow & F (\int^{n} X(n) \otimes Y^{\star n}(k)) \\
& \cong & F(X \circ Y)(k),
\end{array}
\]
where both isomorphisms are induced by universal properties (see \cite{cg} for more details) and the unlabeled arrow is induced by the same argument as that for coproducts above but this time using coends.  The arrow $t$ is constructed in a similar fashion, and is the composite below.
\[
\begin{array}{rcl}
(FY)^{\star n}(k) & = & \int^{k_{1}, \ldots, k_{n}} FY(k_{1}) \otimes \cdots \otimes FY(k_{n}) \odot B\mb{\Lambda}(k, \sum k_{i}) \\
& \rightarrow &  \int^{k_{1}, \ldots, k_{n}} F(Y(k_{1}) \otimes \cdots \otimes Y(k_{n})) \odot B\mb{\Lambda}(k, \sum k_{i}) \\
& \rightarrow & \int^{k_{1}, \ldots, k_{n}} F(Y(k_{1}) \otimes \cdots \otimes Y(k_{n}) \odot B\mb{\Lambda}(k, \sum k_{i}) ) \\
& \rightarrow & F\int^{k_{1}, \ldots, k_{n}} Y(k_{1}) \otimes \cdots \otimes Y(k_{n}) \odot B\mb{\Lambda}(k, \sum k_{i})  \\
& = & F(Y^{\star n}(k))
\end{array}
\]

Checking the lax monoidal functor axioms is tedious but entirely routine using the lax monoidal functor axioms for $F$ together with various universal properties of colimits, and we leave the details to the reader.
\end{proof}

\begin{prop}\label{gisgop}
Let $\mb{\Lambda}$ be an action operad.  Then the operad $\Lambda$ is itself a $\mb{\Lambda}$-operad.
\end{prop}
\begin{proof}
One can in fact easily verify that the two equivariance axioms in the definition of a $\mb{\Lambda}$-operad follow from the final axiom for $\mb{\Lambda}$ being an action operad.
\end{proof}

Combining Theorem \ref{preserveGop} and Proposition \ref{gisgop} with Corollary \ref{symmoncor}, we immediately obtain the following.

\begin{cor}
Let $\mb{\Lambda}$ be an action operad.  Then $E\Lambda = \{ E\big(\Lambda(n)\big) \}$ (see \cref{nota:e_b}) is a $\mb{\Lambda}$-operad in $\mb{Cat}$.
\end{cor}

Any $\mb{\Lambda}$-operad in $\mb{Cat}$ gives rise to a 2-monad on $\mb{Cat}$ \cite{cg}.  In our case, that 2-monad (also denoted $E\Lambda$) is given by
\[
X \mapsto \coprod_{n \geq 0} E\Lambda(n) \times_{\Lambda(n)} X^{n}
\]
where the action of $\Lambda(n)$ on $E\Lambda(n)$ is given by the obvious multiplication action on the right, and the action of $\Lambda(n)$ on $X^{n}$ is given using $\pi_{n}:\Lambda(n) \rightarrow \Sigma_{n}$ together with the standard left action of $\Sigma_{n}$ on $X^{n}$ in any symmetric monoidal category.  It will be useful for our calculations later to give an explicit description of the categories $E\Lambda(n) \times_{\Lambda(n)} X^{n}$.  Objects are equivalence classes of tuples $(g; x_1, \ldots, x_n)$ where $g \in \Lambda(n)$ and the $x_{i}$ are objects of $X$, with the equivalence relation given by
\[
(gh; x_1, \ldots, x_n) \sim (g; x_{\pi(h)^{-1}(1)}, \ldots, x_{\pi(h)^{-1}(n)});
\]
we write these classes as $[g; x_1, \ldots, x_n]$.  Morphisms are then equivalence classes of morphisms
\[
(!; f_1, \ldots, f_n): (g; x_1, \ldots, x_n) \to (g'; x_1', \ldots, x_n').
\]
We have two distinguished classes of morphisms, one for which the map $!: g \to h$ is the identity and one for which all the $f_{i}$'s are the identity.  Every morphism in $E\Lambda(n) \times X^{n}$ is uniquely a composite of a  morphism of the first type followed by one of the second type.  Now $E\Lambda(n) \times_{\Lambda(n)} X^{n}$ is a quotient of $E\Lambda(n) \times X^{n}$ by a free group action, so every morphism of $E\Lambda(n) \times_{\Lambda(n)} X^{n}$ is in the image of the quotient map.  Using this fact, we can prove the following useful lemma.

\begin{lem}\label{hom-set-lemma}
For an action operad $\mb{\Lambda}$ and any category $X$, the set of morphisms from $[e; x_1, \ldots, x_n]$ to $[e; y_1, \ldots, y_n]$ in $E\Lambda(n) \times_{\Lambda(n)} X^{n}$ is
\[
\coprod_{g \in \Lambda(n)} \prod_{i=1}^{n} X(x_i, y_{\pi(g)(i)}).
\]
\end{lem}
\begin{proof}
A morphism with source $(e; x_1, \ldots, x_n)$ in $E\Lambda(n) \times X^{n}$ is uniquely a composite
\[
(e; x_1, \ldots, x_n) \stackrel{(\id; f_{1}, \ldots, f_{n})}{\longrightarrow} (e; x_1', \ldots, x_n') \stackrel{(!; \id, \ldots, \id)}{\longrightarrow} (g; x_1', \ldots, x_n').
\]
Descending to the quotient, this becomes a morphism
\[
[e; x_1, \ldots, x_n] \to [g; x_1', \ldots, x_n'] = [e; x_{\pi(g)^{-1}(1)}', \ldots, x_{\pi(g)^{-1}(n)}'],
\]
and therefore is a morphism $[e; x_1, \ldots, x_n] \to [e; y_1, \ldots, y_n]$ precisely when $y_i = x_{\pi(g)^{-1}(i)}'$, and so $f_i \in 	X(x_i, y_{\pi(g)(i)})$.
\end{proof}

The 2-monad $E\Lambda$ is both finitary and cartesian (see \cite{cg}).  In fact we can characterize this operad uniquely (up to equivalence) using a standard argument.

\begin{Defi}
Let $\mb{\Lambda}$ be an action operad.  A \textit{$\mb{\Lambda}_{\infty}$ operad} $P$ is a $\mb{\Lambda}$-operad in which each action $P(n) \times \Lambda(n) \rightarrow P(n)$ is free and each $P(n)$ is contractible.
\end{Defi}
\begin{rem}
The above definition makes sense in a wide context, but needs interpretation.  We can interpret the freeness condition in any complete category, as completeness allows one to compute fixed points using equalizers.  Contractibility then requires a notion of equivalence or weak equivalence, such as in an $(\infty, 1)$-category or Quillen model category, and a terminal object.  Our interest is in the above definition interpreted in $\mb{Cat}$, in which case both conditions (free action and contractible $P(n)$'s) mean the obvious thing.
\end{rem}

\begin{prop}
$\mb{\Lambda}_{\infty}$ operads in $\mb{Cat}$ are unique up to a zig-zag of pointwise equivalences of $\mb{\Lambda}$-operads.
\end{prop}
\begin{proof}
Given $P,Q$ $\mb{\Lambda}_{\infty}$ operads in $\mb{Cat}$, the product $P \times Q$ with the diagonal action is also $\mb{\Lambda}_{\infty}$.  Each of the projection maps is a pointwise equivalence of $\mb{\Lambda}$-operads.
\end{proof}
\begin{rem}
Once again, this proof holds in a wide context.  We required that the product of free actions is again free, true in any complete category.  We also required that the product of contractible objects is contractible; this condition will hold, for example, in any Quillen model category in which all objects are fibrant or in which the product of weak equivalences is again a weak equivalence.
\end{rem}

\begin{cor}\label{egisinf}
$E\Lambda$ is $\mb{\Lambda}_{\infty}$, hence unique in the sense above.
\end{cor}

\begin{rem}
One should also note that combining Theorem \ref{preserveGop} with Corollary \ref{egisinf}, we get canonical $\mb{\Lambda}_{\infty}$ operads in the category of simplicial sets by taking the nerve (the nerve functor is represented by a cosimplicial category, namely $\Delta \subseteq \mb{Cat}$, so preserves products) and then in suitable categories of topological spaces by taking the geometric realization (once again, product-preserving with the correct category of spaces).  Thus we have something like a Barratt-Eccles $\mb{\Lambda}_{\infty}$ operad for any action operad $\mb{\Lambda}$.
\end{rem}

\section{Abstract properties of the Borel construction}

Kelly's theory of clubs \cite{kelly_club1, kelly_club0, kelly_club2} was designed to simplify and explain certain aspects of coherence results, namely the fact that many coherence results rely on extrapolating information about general free objects for a 2-monad $T$ from information about the specific free object $T1$ where $1$ denotes the terminal category.  This occurs, for example, in the study of the many different flavors of monoidal category:  plain monoidal category, braided monoidal category, symmetric monoidal category, and so on.  This section will explain how every action operad gives rise to a club, as well as compute the clubs which arise as the image of this procedure.

We begin by reminding the reader of the notion of a club, or more specifically what Kelly \cite{kelly_club1,kelly_club2} calls a club over $\mb{P}$.  We will only be interested in clubs over $\mb{P}$, and thusly shorten the terminology to club from this point onward.  Defining clubs is accomplished most succinctly using Leinster's terminology of generalized operads \cite{leinster}.

\begin{Defi}
Let $C$ be a category with finite limits.
\begin{enumerate}
\item A monad $T:C \rightarrow C$ is \textit{cartesian} if the functor $T$ preserves pullbacks, and the naturality squares for the unit $\eta$ and the multiplication $\mu$ for $T$ are all pullbacks.
\item The category of \textit{$T$-collections}, $T\mbox{-}\mb{Coll}$, is the slice category $C/T1$, where $1$ denotes the terminal object.
\item Given a pair of $T$-collections $X \stackrel{x}{\rightarrow} T1, Y \stackrel{y}{\rightarrow} T1$, their \textit{composition product} $X \circ Y$ is given by the pullback below together with the morphism along the top.
    \[
\xy
(0,0)*+{X \circ Y} ="00";
(15,0)*+{TY} ="10";
(30,0)*+{T^{2}1} ="20";
(45,0)*+{T1} ="30";
(0,-10)*+{X} ="01";
(15,-10)*+{T1} ="11";
{\ar^{} "00" ; "10"};
{\ar^{Ty} "10" ; "20"};
{\ar^{\mu} "20" ; "30"};
{\ar^{T!} "10" ; "11"};
{\ar_{x} "01" ; "11"};
{\ar^{} "00" ; "01"};
(3,-3)*{\lrcorner};
\endxy
\]
\item The composition product, along with the unit of the adjunction $\eta:1 \rightarrow T1$, give $T\mbox{-}\mb{Coll}$ a monoidal structure.  A \textit{$T$-operad} is a monoid in $T\mbox{-}\mb{Coll}$.
\end{enumerate}
\end{Defi}

Let $\Sigma$ be the operad of symmetric groups.  This is the terminal object of the category of action operads, with each $\pi_{n}$ the identity map.  Then $E\Sigma$ is a 2-monad on $\mb{Cat}$, and by \cite{cg} it is cartesian.

\begin{Defi}
A \textit{club} is a $T$-operad in $\mb{Cat}$ for $T = E\Sigma$.
\end{Defi}

\begin{rem}
The category $\mb{P}$ in Kelly's terminology is the result of applying $E\Sigma$ to $1$, and can be identified with $B\Sigma = \coprod B\Sigma_{n}$.
\end{rem}

It is useful to break down the definition of a club.  A club consists of
\begin{enumerate}
\item a category $K$ together with a functor $K \rightarrow B \Sigma$,
\item a multiplication map $K \circ K \rightarrow K$, and
\item a unit map $1 \rightarrow K$
\end{enumerate}
satisfying the axioms to be a monoid in the monoidal category of $E\Sigma$-collections.  By the definition of $K \circ K$ as a pullback, we see that objects are tuples of objects of $K$ $(x; y_{1}, \ldots, y_{n})$ where $\pi(x) = n$.  A morphism
\[
(x; y_{1}, \ldots, y_{n}) \to (z; w_{1}, \ldots, w_{m})
\]
exists only when $n=m$ (since $B\Sigma$ only has endomorphisms) and then consists of a morphism $f:x \to z$ in $K$ together with morphisms $g_{i}:y_{i} \to z_{\pi(x)(i)}$ in $K$.

\begin{nota}\label{nota:clubmult}
For a club $K$ and a morphism $(f; g_{1}, \ldots, g_{n})$ in $K \circ K$, we write $f(g_{1}, \ldots, g_{n})$ for the image morphism under the functor $K \circ K \rightarrow K$.
\end{nota}

We will usually just refer to a club by its underlying category $K$.

\begin{thm}
Let $\mb{\Lambda}$ be an action operad.  Then the map of operads $\pi:\Lambda \rightarrow \Sigma$ gives the category $B\mb{\Lambda} = \coprod B\Lambda(n)$ the structure of a club.
\end{thm}
\begin{proof}
To give the functor $B\pi \cn B\mb{\Lambda} \to B \mb{\Sigma}$ the structure of a club it suffices (see \cite{leinster}) to show that
\begin{itemize}
\item the induced monad, which we will show to be $E\Lambda$, is a cartesian monad on $\mb{Cat}$,
\item the transformation $\tilde{\pi}:E\Lambda \Rightarrow E\Sigma$ induced by the functor $E\pi$ is cartesian, and
\item $\tilde{\pi}$ commutes with the monad structures.
\end{itemize}
$E\Lambda$ is always cartesian by results of \cite{cg}.  The transformation $\tilde{\pi}$ is the coproduct of the maps $\tilde{\pi}_{n}$ which are induced by the universal property of the coequalizer as shown below.
\[
\xy
(0,0)*+{\scriptstyle E\Lambda(n) \times \Lambda(n) \times X^n} ="00";
(0,-10)*+{\scriptstyle E\Sigma_{n} \times \Sigma_{n} \times X^n} ="01";
(30,0)*+{\scriptstyle E\Lambda(n) \times X^n} ="10";
(30,-10)*+{\scriptstyle E\Sigma_{n} \times X^n} ="11";
(60,0)*+{\scriptstyle E\Lambda(n) \times_{\Lambda(n)} X^n} ="20";
(60,-10)*+{\scriptstyle E\Sigma_{n} \times_{\Sigma_{n}}  X^n} ="21";
{\ar (12,1)*{}; (22,1)*{} };
{\ar (12,-1)*{}; (22,-1)*{} };
{\ar_{E\pi \times \pi \times 1} "00" ; "01"};
{\ar (12,-9)*{}; (22,-9)*{} };
{\ar (12,-11)*{}; (22,-11)*{} };
{\ar_{E\pi \times 1} "10" ; "11"};
{\ar@{.>}^{\tilde{\pi}_{n}} "20" ; "21"};
{\ar "10" ; "20"};
{\ar "11" ; "21"};
\endxy
\]
Naturality is immediate, and since $\pi$ is a map of operads $\tilde{\pi}$ also commutes with the monad structures.

It only remains to show that $\tilde{\pi}$ is cartesian and that the induced monad is actually $E\Lambda$.  Since the monads $E\Lambda$ and $E\Sigma$ both decompose into a disjoint union of functors, we only have to show that, for any $n$, the square below is a pullback.
\[
\xy
(0,0)*+{E\Lambda(n) \times_{\Lambda(n)} X^n} ="00";
(0,-10)*+{B\Lambda(n)} ="01";
(35,0)*+{E\Sigma_{n} \times_{\Sigma_{n}} X^n} ="10";
(35,-10)*+{B\Sigma_{n}} ="11";
{\ar^{} "00" ; "10"};
{\ar^{} "10" ; "11"};
{\ar^{} "00" ; "01"};
{\ar^{} "01" ; "11"};
\endxy
\]
By the explicit description of the coequalizer given in \cite{cg}, this amounts to showing that the square below is a pullback.
\[
\xy
(0,0)*+{E\Lambda(n) \times X^n/\Lambda(n)} ="00";
(0,-10)*+{B\Lambda(n)} ="01";
(35,0)*+{E\Sigma_{n} \times X^n/\Sigma_{n}} ="10";
(35,-10)*+{B\Sigma_{n}} ="11";
{\ar^{} "00" ; "10"};
{\ar^{} "10" ; "11"};
{\ar^{} "00" ; "01"};
{\ar^{} "01" ; "11"};
\endxy
\]
Here, $A \times B/G$ is the category whose objects are equivalence classes of pairs $(a,b)$ where $(a,b) \sim (ag, g^{-1}b)$, and similarly for morphisms.  Now the bottom map is clearly bijective on objects since these categories only have one object.  An object in the top right is an equivalence class
\[
[\sigma; x_{1}, \ldots, x_{n}] = [e; x_{\sigma^{-1}(1)}, \ldots, x_{\sigma^{-1}(n)}].
\]
A similar description holds for objects in the top left, with $g \in \Lambda(n)$ replacing $\sigma$ and $\pi(g)^{-1}$ replacing $\sigma^{-1}$ in the subscripts.  The map along the top sends $[g; x_{1}, \ldots, x_{n}]$ to $[\pi(g); x_{1}, \ldots, x_{n}]$, and thus sends $[e; x_{1}, \ldots, x_{n}]$ to $[e; x_{1}, \ldots, x_{n}]$, giving a bijection on objects.

Now a morphism in $E\Lambda(n) \times X^{n}/\Lambda(n)$ can be given as
\[
[e; x_{1}, \ldots, x_{n}] \stackrel{[!; f_{i}]}{\longrightarrow} [g; y_{1}, \ldots, y_{n}].
\]
Mapping down to $B\Lambda(n)$ gives $ge^{-1} = g$, while mapping over to $E\Sigma_{n} \times X^{n}/\Sigma_{n}$ gives $[!; f_{i}]$ where $!:e \rightarrow \pi(g)$ is now a morphism in $E\Sigma_{n}$.  In other words, a morphism in the upper left corner of our putative pullback square is determined completely by its images along the top and lefthand functors.  Furthermore, given $g \in \Lambda(n)$, $\tau = \pi(g)$, and morphisms $f_{i}:x_{i} \rightarrow y_{i}$ in $X$, the morphism $[!:e \rightarrow g; f_{i}]$ maps to the pair $(g, [!:e \rightarrow \tau; f_{i}])$, completing the proof that this square is indeed a pullback.
\end{proof}

The club, which we now denote $K_{\mb{\Lambda}}$, associated to $E\Lambda$ has the following properties.  First, the functor $K_{\mb{\Lambda}} \rightarrow B\Sigma$ is a functor between groupoids.  Second, the functor $K_{\mb{\Lambda}} \rightarrow B\Sigma$ is  bijective-on-objects.  We claim that these properties characterize those clubs which arise from action operads.  Thus the clubs arising from action operads are very similar to PROPs \cite{mac_prop, markl_prop}.

\begin{thm}\label{thm:club=operad}
Let $K$ be a club such that
\begin{itemize}
\item the map $K \rightarrow B \mb{\Sigma}$ is bijective on objects and
\item $K$ is a groupoid.
\end{itemize}
Then $K \cong K_{\mb{\Lambda}}$ for some action operad $\mb{\Lambda}$.  The assignment $\mb{\Lambda} \mapsto K_{\mb{\Lambda}}$ is a full and faithful embedding of the category of action operads $\mb{AOp}$ into the category of clubs.
\end{thm}
\begin{proof}
Let $K$ be such a club.  Our hypotheses immediately imply that $K$ is a groupoid with objects in bijection with the natural numbers; we will now assume the functor $K \to B\mb{\Sigma}$ is the identity on objects.  Let $\Lambda(n) = K(n,n)$.  Now $K$ comes equipped with a functor to $B \mb{\Sigma}$, in other words group homomorphisms $\pi_{n}:\Lambda(n) \rightarrow \Sigma_{n}$.  We claim that the club structure on $K$ makes the collection of groups $\{ \Lambda(n) \}$ an action operad.  In order to do so, we will employ \cref{thm:charAOp}.

First, we give the group homomorphism $\beta$ using \cref{nota:clubmult}.  Define 
\[
\beta(g_{1}, \ldots, g_{n}) = e_{n}(g_{1}, \ldots, g_{n})
\]
 (see \ref{nota:clubmult}) where $e_{n}$ is the identity morphism $n \to n$ in $K(n,n)$.  Functoriality of the club multiplication map immediately implies that this is a group homomorphism.  Second, we define the function $\delta$ in a similar fashion:
\[
\delta_{n; k_{1}, \ldots, k_{n}}(f) = f(e_{1}, \ldots, e_{n}),
\]
where here $e_{i}$ is the identity morphism of $k_{i}$ in $K$.

There are now nine axioms to verify in \cref{thm:charAOp}.  The club multiplication functor is a map of collections, so a map over $B\mb{\Sigma}$; this fact immediately implies that axioms \eqref{eq1} (using morphisms in $K \circ K$ with only $g_{i}$ parts) and \eqref{eq4} (using morphisms in $K \circ K$ with only $f$ parts) hold.  The mere fact that multiplication is a functor also implies axioms \eqref{eq6} (once again using morphisms with only $f$ parts) and \eqref{eq8} (by considering the composite of a morphism with only an $f$ with a morphism with only $g_{i}$'s).  Axiom \eqref{eq2} is the equation $e_{1}(g) = g$ which is a direct consequence of the unit axiom for the club $K$; the same is true of axiom \eqref{eq5}.  Axioms \eqref{eq3}, \eqref{eq7},  and \eqref{eq9} all follow from the associativity of the club multiplication.

Finally, we would like to show that this gives a full and faithful embedding $K_{-}:\mb{AOp} \rightarrow \mb{Club}$ of the category of action operads into the category of clubs.  Let $f, f':\mb{\Lambda} \rightarrow \mb{\Lambda'}$ be maps between action operads.  Then if $K_{f} = K_{f'}$ as maps between clubs, then they must be equal as functors $K_{\mb{\Lambda}} \rightarrow K_{\mb{\Lambda'}}$.  But these functors are nothing more than the coproducts of the functors
\[
B(f_{n}), B(f_{n}'):B\Lambda(n) \rightarrow B\Lambda'(n),
\]
and the functor $B$ from groups to categories is faithful, so $K_{-}$ is also faithful.  Now let $f:K_{\mb{\Lambda}} \rightarrow K_{\mb{\Lambda'}}$ be a maps of clubs.  We clearly get group homomorphisms $f_{n}:\Lambda(n) \rightarrow \Lambda'(n)$ such that $\pi^{\Lambda}_{n} = \pi^{\Lambda'}_{n} f_{n}$, so we must only show that the $f_{n}$ also constitute an operad map.  Using the description of the club structure above in terms of the maps $\beta, \delta$, we see that commuting with the club multiplication implies commuting with both of these, which in turn is equivalent to commuting with operad multiplication.  Thus $K_{-}$ is full as well.
\end{proof}

\begin{rem}
First, one should note that being a club over $B \mb{\Sigma}$ means that every $K$-algebra has an underlying strict monoidal structure.  Second, requiring that $K \rightarrow B \mb{\Sigma}$ be bijective on objects ensures that $K$ does not have  operations other than $\otimes$, such as duals or internal hom-objects, from which to build new types of objects.  Finally, $K$ being a groupoid ensures that all of the ``constraint morphisms'' that exist in algebras for $K$ are invertible.

These hypotheses could be relaxed somewhat.  Instead of having a club over $B \mb{\Sigma}$, we could have a club over the free symmetric monoidal category on one object (note that the free symmetric monoidal category monad on $\mb{Cat}$ is still cartesian).  This would produce $K$-algebras with underlying monoidal structures which are not necessarily strict.  This change should have relatively little impact on how the theory is developed.  Changing $K$ to be a category instead of a groupoid would likely have a larger impact, as the resulting action operads would have monoids instead of groups at each level.  We have made repeated use of inverses throughout the proofs in the basic theory of action operads, and these would have to be revisited if groups were replaced by monoids in the definition of action operads.
\end{rem}

\section{Presentations of action operads}

One of the most useful methods for constructing new examples of some given algebraic structure is through the use of presentations.  A presentation consists of generating data together with relations between generators using the operations of the algebra involved.  In categorical terms, the generators  and relations are both given as free gadgets on some underlying data, and the presentation itself is a coequalizer.  This section will establish the categorical structure necessary to give presentations for action operads, and then explain how such a presentation is reflected in the associated club and 2-monad.  The most direct route to the desired results uses the theory of locally finitely presentable categories.  We recall the main definitions briefly, but refer the reader to \cite{ar} for additional details.

\begin{Defi}
A \textit{filtered category} is a nonempty category $C$ such that
\begin{itemize}
\item if $a,b$ are objects of $C$, then there is another object $c \in C$ and morphisms $a \to c, b \to c$; and
\item if $f,g \cn a \to b$ are parallel morphisms in $C$, then there exists a morphism $h \cn b \to c$ such that $hf = hg$.
\end{itemize}
\end{Defi}

\begin{Defi}
A \emph{filtered colimit} is a colimit over a filtered category.
\end{Defi}

\begin{Defi}
Let $C$ be a category with all filtered colimits.  An object $x \in C$ is \textit{finitely presentable} if the representable functor $C(x, -):C \rightarrow \mb{Sets}$ preserves filtered colimits.
\end{Defi}

\begin{Defi}
A \textit{locally finitely presentable category} is a category $C$ such that
\begin{itemize}
\item $C$ is cocomplete and
\item there exists a small subcategory $C_{fp} \subseteq C$ of finitely presentable objects such that any object $x \in C$ is the filtered colimit of some diagram in $C_{fp}$.
\end{itemize}
\end{Defi}

The definition of a locally finitely presentable category has many equivalent variants, and our applications are quite straightforward so we have given what is probably the most common version of the definition.  We refer the reader to \cite{ar} for more information.  The following result will be clear to the experts, so we just sketch a proof.

\begin{thm}
The category $\mb{AOp}$ is locally finitely presentable.
\end{thm}
\begin{proof}
First, note that there is a category $\mb{Op}^{g}$ whose objects are operads $P$ in which each $P(n)$ also carries a group structure.  This is an equational theory using equations with only finitely many elements, so $\mb{Op}^{g}$ is locally finitely presentable.  The symmetric operad is an object of this category, so the slice category $\mb{Op}^{g}/\Sigma$ is also locally finitely presentable.  There is an obvious inclusion functor $\mb{AOp} \hookrightarrow \mb{Op}^{g}/\Sigma$.  Now $\mb{AOp}$ is a full subcategory of $\mb{Op}^{g}/\Sigma$ which is closed under products, subobjects, and any object of $\mb{Op}^{g}/\Sigma$ isomorphic to an action operad is in fact an action operad, so the inclusion   $\mb{AOp} \hookrightarrow \mb{Op}^{g}/\Sigma$ is actually the inclusion of a reflective subcategory.  One easily checks that $\mb{AOp}$ is in fact closed under all limits and filtered colimits in $\mb{Op}^{g}/\Sigma$, so by the Reflection Theorem (2.48 in \cite{ar}), $\mb{AOp}$ is locally finitely presentable. 
\end{proof}

\begin{Defi}
Let $\SS$ be the set which is the disjoint union of the underlying sets of all the symmetric groups.  Then $\mb{Sets}/\SS$ is the slice category over $\SS$ with objects $(X,f)$ where $X$ is a set and $f:X \rightarrow \SS$ and morphisms $(X_{1}, f_{1}) \rightarrow (X_{2}, f_{2})$ are those functions $g:X_{1} \rightarrow X_{2}$ such that $f_{1} = f_{2}g$.  We call an object $(X,f)$ a \textit{collection over $\SS$}.
\end{Defi}

\begin{rem}
In standard presentations of the theory of operads (see, for example, \cite{mss-op}), a nonsymmetric operad will have an underlying collection (or $\mathbb{N}$-indexed collection of sets) while a symmetric operad will have an underlying symmetric collection (or $\mathbb{N}$-indexed collection of sets in which the $n$th set has an action of $\Sigma_{n}$).  Our collections over $\SS$ more closely resemble the former as there is no group action present.
\end{rem}

\begin{thm}\label{underlyingSS}
There is a forgetful functor $U:\mb{AOp} \rightarrow \mb{Sets}/\SS$ which preserves all limits and filtered colimits.
\end{thm}
\begin{proof}
The functor $U$ is obvious, and the preservation of filtered colimits follows from the fact that these are computed pointwise, together with the fact that every map between action operads preserves underlying permutations.  For limits, recall that limits in $\mb{AOp}$ are computed as in the category of operads over $\Sigma$.  This means that equalizers are computed levelwise, and the product $\mb{\Lambda} \times \mb{\Lambda}'$ has underlying operad the pullback $\Lambda \times_{\Sigma} \Lambda'$; this pullback is itself computed levelwise.  Together, these imply that $U$ also preserves all limits.
\end{proof}

\begin{cor}
$U$ has a left adjoint $F:\mb{Sets}/\SS \rightarrow \mb{AOp}$, the free action operad functor.
\end{cor}
\begin{proof}
The category $\mb{Sets}/\SS$ is locally finitely presentable as it is equivalent to the functor category $[\SS, \mb{Sets}]$ (here $\SS$ is treated as a discrete category) and any presheaf category is locally finitely presentable.  The functor $U$ preserves limits and filtered colimits between locally finitely presentable categories, so has a left adjoint (see Theorem 1.66 in \cite{ar}).
\end{proof}

\begin{Defi}
A \textit{presentation} for an action operad $\mb{\Lambda}$ consists of
\begin{itemize}
\item a pair of collections over $\SS$ denoted $\mathbf{g}, \mathbf{r}$,
\item a pair of maps $s_{1}, s_{2}:F\mathbf{r} \rightarrow F\mathbf{g}$ between the associated free action operads, and
\item a map $p:F\mathbf{g} \rightarrow \mb{\Lambda}$ of action operads exhibiting $\mb{\Lambda}$ as the coequalizer of $s_{1},s_{2}$.
\end{itemize}
\end{Defi}

In \cite{kelly_club1}, Kelly discusses clubs given by generators and relations.  His generators include functorial operations more general than what we are interested in here, and the natural transformations are not required to be invertible.  In our case, the only generating operations we require are those of a unit and tensor product, as the algebras for $E\Lambda$ are always strict monoidal categories with additional structure.  Tracing through his discussion of generators and relations for a club gives the following theorem.

\begin{thm}\label{pres1}
Let $\mb{\Lambda}$ be an action operad with presentation given by $(\mathbf{g},\mathbf{r}, s_{i}, p)$.  Then the club $E\Lambda$ is generated by
\begin{itemize}
\item functors giving the unit object and tensor product, and
\item natural transformations given by the collection $\mathbf{g}$:  each element $x$ of $\mathbf{g}$ with $\pi(x) = \sigma_{x} \in \Sigma_{|x|}$ gives a natural transformation from the $n$th tensor power functor to itself,
\end{itemize}
subject to relations such that the following axioms hold.
\begin{itemize}
\item The monoidal structure given by the unit and tensor product is strict.
\item The transformations given by the elements of $\mathbf{g}$ are all natural isomorphisms.
\item For each element $y \in \mathbf{r}$, the equation $s_{1}(y) = s_{2}(y)$ holds.
\end{itemize}
\end{thm}

Bringing this down to a concrete level we have the following corollary.

\begin{cor}\label{pres2}
Assume we have a notion $\mathcal{M}$ of strict monoidal category which is given by  a set natural isomorphisms
\[
\mathcal{G} = \{ (f, \pi_{f}) \, | \,  x_{1} \otimes \cdots \otimes x_{n} \stackrel{f}{\longrightarrow} x_{\pi_{f}^{-1}(1)} \otimes \cdots \otimes x_{\pi_{f}^{-1}(n)} \}
\]
subject to a set $\mathcal{R}$ of axioms.  Each such axiom is given by the data
\begin{itemize}
\item two finite sets $f_{1}, \ldots, f_{n}$ and $f_{1}', \ldots, f_{m}'$ of elements of $\mathcal{G}$; and
\item two formal composites $F,F'$ using only composition and tensor product operations and the $f_{i}$, respectively $f_{i}'$, 
\end{itemize}
such that the underlying permutation of $F$ equals the underlying permutation of $F'$ (we compute the underlying permutations using the functions $\beta, \delta$ of \cref{thm:charAOp}).  The element $(\underline{f}, \underline{f}', F, F')$ of the set $\mathcal{R}$ of axioms corresponds to the requirement that the composite of the morphisms $f_{i}$ using $F$ equals the composite of the morphisms $f_{j}'$ using $F'$ in any strict monoidal category of type $\mathcal{M}$.  Then strict monoidal categories of type $\mathcal{M}$ are given as the algebras for the club $E\Lambda$ where $\mb{\Lambda}$ is the action operad with
\begin{itemize}
\item $\mathbf{g} = \mathcal{G}$,
\item $\mathbf{r} = \mathcal{R}$,
\item $s_{1}$ given by mapping the generator $(\underline{f}, \underline{f}', F, F')$ to the operadic composition of the $f_{i}$ using $F$ via $\beta, \delta$, and
\item $s_{2}$ given by mapping the generator $(\underline{f}, \underline{f}', F, F')$ to the operadic composition of the $f_{i}'$ using $F'$ via $\beta, \delta$.
\end{itemize}
\end{cor}

\section{Examples}

In this section we will discuss examples of the preceding theory.  We have seen that there are three equivalent incarnations of the same algebraic structure:
\begin{itemize}
\item as an action operad $\mb{\Lambda}$,
\item as a 2-monad $X \mapsto \coprod E\Lambda(n) \times_{\Lambda(n)} X^{n}$ on $\mb{Cat}$, or
\item as a club $B\Lambda \rightarrow B\Sigma$ satisfying certain properties.
\end{itemize}
In practice, something like the third of these is the most likely to arise from applications (even if the notion of a club is perhaps less well-known outside of the categorical literature than that of an operad or a 2-monad) as a club can be given by a presentation as we discussed in Section 3.  We will go into more detail here, explaining how particular monoidal structures of interest are represented in this theory.

\begin{example}
The 2-monad for symmetric strict monoidal categories (or permutative categories, as they are known in the topological literature) is given by $E \mb{\Sigma}$, so the notion of symmetric strict monoidal categories corresponds to the symmetric operad.  While this example is well-known, we go into further detail to set the stage for less common examples.

The 2-monad $E\Sigma$ on $\mb{Cat}$ is given by
\[
E\Sigma (X) = \coprod E\Sigma_{n} \times_{\Sigma_{n}} X^{n}.
\]
An object of $E\Sigma_{n} \times_{\Sigma_{n}} X^{n}$ is an equivalence class of the form $[\sigma; x_{1}, \ldots, x_{n}]$ where $\sigma \in \Sigma_{n}$ and $x_{i} \in X$.  The equivalence relation gives
\[
[\sigma; x_{1}, \ldots, x_{n}] = [e; x_{\sigma^{-1}(1)}, \ldots, x_{\sigma^{-1}(n)}],
\]
so objects can be identified with finite strings of objects of $X$.  Morphisms are given by equivalence classes of the form
\[
[\sigma; x_{1}, \ldots, x_{n}] \stackrel{[!; f_{1}, \ldots, f_{n}]}{\longrightarrow} [\tau; y_{1}, \ldots, y_{n}].
\]
Here $!:\sigma \cong \tau$ is the unique isomorphism in $E \Sigma_{n}$, and $f_{i}:x_{i} \rightarrow y_{i}$ in $X$.  Using the equivalence relation, we get that morphisms between finite strings
\[
x_{1}, \ldots, x_{n} \rightarrow y_{1}, \ldots, y_{n}
\]
are given by a permutation $\rho \in \Sigma_{n}$ together with maps $f_{i}:x_{i} \rightarrow y_{\rho(i)}$ in $X$ (note that there are no morphisms between strings of different length); this is a special case of the calculation in \ref{hom-set-lemma}.  Thus $E \Sigma(X)$ is easily seen to be the free permutative category generated by $X$, and therefore $E \Sigma$-algebras are permutative categories.
\end{example}

\begin{example}
The template above can be used to show that the braid operad $\mb{B}$ corresponds to the 2-monad for braided strict monoidal categories.  The details are almost exactly the same, only we use braids instead of permutations.  The equivalence relation on objects gives
\[
[\gamma; x_{1}, \ldots, x_{n}] = [e; x_{\pi(\gamma)^{-1}(1)}, \ldots, x_{\pi(\gamma)^{-1}(n)}],
\]
where $\gamma \in B_{n}$ and $\pi(\gamma)$ is its underlying permutation; thus objects of $EB(X)$ are once again finite strings of objects of $X$.  A morphism
\[
x_{1}, \ldots, x_{n} \rightarrow y_{1}, \ldots, y_{n}
\]
is then given by a braid $\gamma \in B_{n}$ together with maps $f_{i}:x_{i} \rightarrow y_{\pi(\gamma)(i)}$ in $X$.  Thus one should view a morphism as given by
\begin{itemize}
\item a finite ordered set $x_{1}, \ldots, x_{n}$ of objects of $X$ as the source,
\item another such finite ordered set (of the same cardinality) $y_{1}, \ldots, y_{n}$ of objects of $X$ as the target,
\item a geometric braid $\gamma \in B_{n}$ on $n$ strands, and
\item for each strand, a morphism in $X$ from the object labeling the source of that strand to the object labeling the target.
\end{itemize}
This is precisely Joyal and Street's \cite{js} construction of the free braided strict monoidal category generated by a category $X$, and thus we see that $EB$-algebras are braided strict monoidal categories.

This example can be extended to include ribbon braided categories as well.  A \textit{ribbon braid} is given, geometrically, in much the same way as a braid except that instead of paths $[0,1] \rightarrow \mathbb{R}^{3}$ making up each individual strand, we use ribbons
$[0,1] \times [-\varepsilon, \varepsilon] \rightarrow \mathbb{R}^{3}$.  This introduces the possibility of performing a full twist on a ribbon, and one can describe ribbon braided categories using generators and relations by introducing a natural twist isomorphism $\tau_{A}:A \rightarrow A$ and imposing one relation between the twist and the braid $\gamma_{A,B}:A \otimes B \rightarrow B \otimes A$.  In \cite{sal-wahl}, the authors show that the ribbon braid groups give an action operad $\mb{RB}$, and that (strict) ribbon braided categories are precisely the algebras for $ERB$.
\end{example}

We now turn to an example which is not as widely known in the categorical literature, that of coboundary categories \cite{drin-quasihopf}.  These arise in the representation theory of quantum groups and in the theory of crystals \cite{hk-cobound, hk-quantum}.  Our goal here is to refine the relationship between coboundary categories and the operad of $n$-fruit cactus groups in \cite{hk-cobound} by using the theory of action operads and our Borel construction.  We begin by recalling the definition of a coboundary category.

\begin{Defi}
A \textit{coboundary category} is a monoidal category $C$ equipped with a natural isomorphism $\sigma_{x,y}:x \otimes y \rightarrow y \otimes x$ (called the \textit{commutor}) such that
\begin{itemize}
\item $\sigma_{y,x} \circ \sigma_{x,y} = 1_{x \otimes y}$ and
\item the diagram
\[
\xy
(0,0)*+{(x \otimes y) \otimes z} ="00";
(35,0)*+{x \otimes (y \otimes z)} ="10";
(70,0)*+{x \otimes (z \otimes y)} ="20";
(0,-15)*+{(y \otimes x) \otimes z} ="01";
(35,-15)*+{z \otimes (y \otimes x)} ="11";
(70,-15)*+{(z \otimes y )\otimes x} ="21";
{\ar "00"; "10" };
{\ar^{1 \sigma_{y,z}} "10"; "20" };
{\ar^{\sigma_{x,zy}} "20"; "21" };
{\ar_{\sigma_{x,y}1} "00"; "01" };
{\ar_{\sigma_{yx,z}} "01"; "11" };
{\ar "11"; "21" };
\endxy
\]
commutes (in which the unlabeled morphisms are an associator and an inverse associator).
\end{itemize}
\end{Defi}

\begin{example}
\begin{enumerate}
\item From the definition above, it is clear that any symmetric monoidal category is also a coboundary category.
\item The name coboundary category comes from the original work of Drinfeld \cite{drin-quasihopf} in which he shows that the category of representations of coboundary Hopf algebra has the structure of coboundary category.
\item Henriques and Kamnitzer \cite{hk-cobound} show that the category of crystals for a finite dimensional complex reductive Lie algebra has the structure of a coboundary category. 
\end{enumerate}
\end{example}

Our interest is in strict coboundary categories by which we mean coboundary categories with strict underlying monoidal category.  Under the assumption of strictness, the second axiom above does not include associations for the tensor product and reduces to a square.  To show that every coboundary category is equivalent to a strict coboundary category, we must introduce the 2-category $\mb{CobCat}$ of coboundary categories.

\begin{Defi}
Let $(C,\sigma), (C', \sigma')$ be coboundary categories.  A \emph{coboundary functor} $F:C \rightarrow C'$ is a strong monoidal functor (with invertible constraints $\varphi_{0}$ for the unit and $\varphi_{x,y}$ for the tensor product) such that
\[
F\sigma_{x,y} \circ \varphi_{x,y} = \varphi_{y,x} \circ \sigma_{Fx,Fy}'
\]
holds.
\end{Defi}

Coboundary functors are composed just as strong monoidal functors are, giving the following.

\begin{lem}
There is a 2-category $\mb{CobCat}$  of coboundary categories, coboundary functors, and monoidal transformations.
\end{lem}

\begin{prop}
Let $(C, \sigma)$ be a coboundary category.  Then there is a strict coboundary category $(C', \sigma')$ which is equivalent to $(C, \sigma)$ in $\mb{CobCat}$.
\end{prop}
\begin{proof}
Consider the underlying monoidal category of $(C, \sigma)$, which we will just write as $C$.  We can find a strict monoidal category $C'$ by coherence for monoidal categories together with an equivalence, as monoidal categories, between $C$ and $C'$.  By standard methods, this can be improved to an adjoint equivalence between $C$ and $C'$ in the 2-category of monoidal categories, strong monoidal functors, and monoidal transformations.  Let $F: C \rightarrow C', G:C' \rightarrow C$ be the functors in this adjoint equivalence, and $\eta: 1 \Rightarrow FG$ the unit (which we note for emphasis is invertible).  For objects $x,y \in C'$, we define a commutor $\sigma'$ for $C'$ as the composite
\[
\begin{array}{rcl}
xy & \stackrel{\eta \otimes \eta}{\longrightarrow} & FGxFGy \\
& \cong & F(GxGy) \\
& \stackrel{F\sigma}{\longrightarrow} & F(GyGx) \\
& \cong  & FGyFGx \\
& \stackrel{\eta^{-1} \otimes \eta^{-1}}{\longrightarrow} & yx.
\end{array}
\]
We then leave to the reader the computations to show that $\sigma'$ is a commutor for $C'$ and that $F,G$ become coboundary functors using $\sigma'$.
\end{proof}

We now turn to the operadic description of strict coboundary categories; we note from this point onwards, all our coboundary categories are assumed to be strict.

\begin{Defi}
Fix $n>1$, and let $1 \leq p < q \leq n$, $1 \leq k < l \leq n$.
\begin{enumerate}
\item $p<q$ is \textit{disjoint} from $k<l$ if $q<k$ or $l<p$.
\item $p<q$ \textit{contains} $k<l$ if $p \leq k < l \leq q$.
\end{enumerate}
\end{Defi}

\begin{Defi}
Let $1 \leq p < q \leq n$, and define $\hat{s}_{p,q} \in \Sigma_{n}$ to be the permutation defined below.
\[
\begin{array}{r|ccccccccccccc}
i & 1 & 2 & \cdots & p-1 & p & p+1 & p+2 & \cdots & q-1 & q & q+1 & \cdots & n \\
\hat{s}_{p,q}(i) & 1 & 2 & \cdots & p-1 & q & q-1 & q-2 & \cdots & p+1 & p & q+1 & \cdots & n
\end{array}
\]
\end{Defi}

\begin{Defi}\label{Defi:defcactus}
Let $J_{n}$ be the group generated by symbols $s_{p,q}$ for $1 \leq p < q \leq n$ subject to the following relations.
\begin{enumerate}
\item For all $p < q$, $s_{p,q}^{2}=e$.
\item If $p<q$ is disjoint from $k<l$, then $s_{p,q}s_{k,l} = s_{k,l}s_{p,q}$.
\item If $p<q$ contains $k<l$, then $s_{p,q}s_{k,l} = s_{m,n}s_{p,q}$ where
\begin{itemize}
\item $m = \hat{s}_{p,q}(l)$ and
\item $n = \hat{s}_{p,q}(k)$.
\end{itemize}
\end{enumerate}
\end{Defi}

It is easy to check that the elements $\hat{s}_{p,q} \in \Sigma_{n}$ satisfy the three relations in  \cref{Defi:defcactus}, so $s_{p,q} \mapsto \hat{s}_{p,q}$ extends to a group homomorphism $\pi_{n}:J_{n} \rightarrow \Sigma_{n}$.  This is the first step in proving the following.

\begin{thm}\label{J_aop}
The collection of groups $J = \{ J_{n} \}$ form an action operad.
\end{thm}
\begin{proof}
We will use \cref{thm:charAOp} to determine the rest of the action operad structure.  Thus we must give, for any collection of natural numbers $n, k_{1}, \ldots, k_{n}$ and $\underline{k} = \sum k_{i}$, group homomorphisms $\beta:J_{k_{1}} \times \cdots \times J_{k_{n}} \rightarrow J_{\underline{k}}$ and functions $\delta: J_{n} \rightarrow J_{\underline{k}}$ satisfying nine axioms.    We define both of these on generators, starting with $\beta$.

Let $s_{p_{i}, q_{i}} \in J_{k_{i}}$.  Let $r_{i} = k_{1} + k_{2} + \cdots + k_{i-1}$.  Define $\beta$ by
\[
\beta(s_{p_{1}, q_{1}}, \ldots, s_{p_{n}, q_{n}}) = s_{p_{1}, q_{1}} s_{p_{2}+r_{2}, q_{2}+r_{2}} \cdots s_{p_{n}+r_{n}, q_{n}+r_{n}}.
\]
Note that $s_{p_{i}+r_{i}, q_{i}+r_{i}}$ and $s_{p_{j}+r_{j}, q_{j}+r_{j}}$ are disjoint when $i \neq j$.
%
It is easy to check that this disjointness property ensures that $\beta$ gives a well-defined group homomorphism
\[
J_{k_{1}} \times \cdots \times J_{k_{n}} \rightarrow J_{\underline{k}}.
\]

To define $\delta: J_{n} \rightarrow J_{\underline{k}}$ for natural numbers $n, k_{1}, \ldots, k_{n}$ and $\underline{k} = \sum k_{i}$, let $m_{k} = s_{1,k} \in J_{k}$.  Then we start by defining
\[
\delta(m_{n}) = m_{\underline{k}} \cdot \beta(m_{k_{1}}, m_{k_{2}}, \ldots, m_{k_{n}}).
\]
Note that, by the containment relation, this is equal to
\[
\beta(m_{k_{n}}, m_{k_{n-1}}, \ldots, m_{k_{1}}) \cdot m_{\underline{k}}.
\]
Now $s_{p,q} \in J_{n}$ is equal to $\beta(e_{p-1}, m_{q-p+1}, e_{n-q})$ (here $e_{i}$ is the identity element in $J_{i}$) by definition of the $m_{i}$ and $\beta$, so we can define $\delta$ on any generator $s_{p,q}$ by
\[
\delta(s_{p,q}) = \beta ( e_{A}, M, e_{B} )
\]
with
\begin{itemize}
\item $A = k_{1} + k_{2} + \cdots + k_{p-1}$,
\item $M = m_{k_{p}+ \cdots +k_{q}} \cdot \beta(m_{k_{p}}, m_{k_{p+1}}, \ldots, m_{k_{q}})$, and
\item $B = k_{q+1} + k_{q+2} + \cdots + k_{n}$.
\end{itemize}
Unpacking this yields the following formula:
\[
\resizebox{\textwidth}{!}{$\delta(s_{p,q}) = s_{k_{1}+\cdots+k_{p-1}+1, k_{1}+\cdots+k_{q}} \cdot \beta(e_{k_{1}+\cdots+k_{p-1}}, m_{k_{p}}, \ldots, m_{k_{q}}, e_{k_{q+1}+\cdots+k_{n}}).$}
\]

We extend $\delta$ to products of generators using axiom 6 of \cref{thm:charAOp}.  As before, we must check that this gives a well-defined function on products of two generators in each of the relations of the cactus groups, and we must also check that this is well-defined on products of three or more generators.  Thus we define
\[
\delta_{n; j_{i}}(gh) = \delta_{n; k_{i}}(g)\delta_{n; j_{i}}(h)
\]
where $k_{i} = j_{\pi(h)^{-1}(i)}$.  There are three relations we must verify for compatibility.
\begin{itemize}
\item We must show that $\delta_{n; j_{i}}(s_{p,q}^{2}) = e$.  By definition, we have
\[
\delta_{n; j_{i}}(s_{p,q}^{2}) = \delta_{n; k_{i}}(s_{p,q})\delta_{n; j_{i}}(s_{p,q})
\]
which is
\[
m_{\underline{j}}\beta(m_{j_{n}}, \ldots, m_{j_{1}}) m_{\underline{j}} \beta(m_{j_{1}}, \ldots, m_{j_{n}}).
\]
By the remarks above in the definition of $\delta$ and the fact that $s_{p,q}^{2}=e$, the element above is easily seen to be the identity.
\item We must show that $\delta(s_{p,q}s_{k,l}) = \delta(s_{k,l}s_{p,q})$ when $(p,q)$ is disjoint from $(k,l)$.  This is another simple calculation using the definition of $\delta$ and the disjointness of the terms involved.
\item We must show that $\delta(s_{p,q}s_{k,l}) = \delta(s_{a,b}s_{p,q})$,  where $a = \hat{s}_{p,q}(l), b = \hat{s}_{p,q}(k)$, if $p < k < l < q$.  In this case, we use all of the relations in the cactus groups to show that each side is equal to
\[
\resizebox{\textwidth}{!}{$\beta(e_{j_{1}}, \ldots, e_{j_{p-1}}, m_{j_{p}+\cdots + j_{q}} \cdot \beta (m_{j_{p}}, \ldots m_{j_{k-1}}, m_{j_{k}+ \cdots j_{l}}, m_{j_{l+1}}, \cdots, m_{j_{q}}), m_{j_{q+1}}, \ldots, m_{j_{n}}).$}
\]
\end{itemize}
In order to show that this gives a well-defined function on products of three or more generators, one proceeds inductively to show that $\delta\big((fg)h\big) = \delta\big(f(gh)\big)$ using the formula above.  This is simply a matter of keeping track of the permutations used to define the subscripts for the different $\delta$'s and we leave it to the reader.  This concludes the definition of the family of functions $\delta_{n; j_{i}}$.

There are now nine axioms to check in \cref{thm:charAOp}.  Axioms \eqref{eq1} - \eqref{eq3} all concern $\beta$, and are immediate from the defining formula.  Axiom \eqref{eq4} is obvious for the elements $m_{k}$, from which it follows in general by the formulas defining $\delta$.  For axiom \eqref{eq5}, one can check easily that
\[
\delta_{n; 1, \ldots, 1}(m_{n}) = m_{n}, \quad \delta_{1;n}(m_{n}) = m_{n}
\]
and once again the general case follows from these.  Axiom \eqref{eq6} holds by the construction of $\delta$.  Axiom \eqref{eq8} can be verified with only one $h_{i}$ nontrivial at a time, and then it is a simple consequence of the second and third relations for $J_{n}$.

Axiom \eqref{eq9} is straightforward to check when only a single $g_{i}$ is a generator and the rest are identities using the defining formulas, and the general case then follows using axiom \eqref{eq6}.  Using \eqref{eq9}, we can then prove axiom \eqref{eq7} as follows; we suppress the subscripts on different $\delta$'s for clarity.  We must show
\[
\delta_{m_1 + \cdots + m_n; p_{11}, \ldots, p_{1m_{1}}, p_{21}, \ldots, p_{nm_{m}}}\big( \delta_{n; m_{1}, \ldots, m_{n}}(f) \big) = \delta_{n; P_{1}, \ldots, P_{n}}(f),
\]
and we do so on $m_{n}$.  By definition, we have
\[
\delta \big( \delta(m_{n}) \big) = \delta \big( m_{\underline{k}} \beta(m_{k_{1}}, \ldots, m_{k_{n}}) \big),
\]
which by axiom \eqref{eq6} is equal to
\[
m_{P_{1} + \cdots P_{n}} \cdot \beta(m_{p_{11}}, \ldots, m_{p_{n,m_{n}}}) \cdot \delta\big( \beta(m_{k_{1}}, \ldots, m_{k_{n}}) \big).
\]
Now this last term is equal to $\beta \big( \delta(m_{k_{1}}), \ldots, \delta(m_{k_{n}}) \big)$ by axiom \eqref{eq9}, which is then equal to
\[
\beta \big( m_{P_{1}}\cdot \beta(m_{p_{11}}, \ldots, m_{p_{1,m_{1}}}), \ldots,  m_{P_{n}}\cdot \beta(m_{p_{n1}}, \ldots, m_{p_{1,m_{n}}}) \big).
\]
Taken all together, the left hand side of axiom \eqref{eq9} is then
\[
\resizebox{\textwidth}{!}{$m_{P_{1} + \cdots P_{n}} \cdot \beta(m_{p_{11}}, \ldots, m_{p_{n,m_{n}}}) \cdot \beta \big( m_{P_{1}}\cdot \beta(m_{p_{11}}, \ldots, m_{p_{1,m_{1}}}), \ldots,  m_{P_{n}}\cdot \beta(m_{p_{n1}}, \ldots, m_{p_{1,m_{n}}}) \big).$}
\]
All of the terms coming from an $m_{p_{ij}}$ can be collected together, and since $s_{p,q}^{2} = e$ for all $p,q$, these cancel.  This leaves
\[
m_{P_{1} + \cdots P_{n}} \cdot \beta \big( m_{P_{1}}, \ldots,  m_{P_{n}} \big)
\]
which is the right hand side of axiom \eqref{eq9} as desired.

\end{proof}

\begin{lem}
The 2-monad $C$ for strict coboundary categories is a club.
\end{lem}
\begin{proof}
This is obvious by \ref{pres2}.
\end{proof}

\begin{thm}
The free coboundary category on one element, $C1$, is isomorphic to $B\mb{J} = \coprod BJ_{n}$.
\end{thm}
\begin{proof}
The universal property we desire is with respect to strict coboundary functors (i.e., coboundary functors whose underlying monoidal functor is strict), so we must give $B\mb{J}$ the structure of a strict coboundary category and then check that to give a strict coboundary functor $B\mb{J} \to X$ to any other strict coboundary category is the same as giving an object of $X$.

The category $B\mb{J}$ has natural numbers as objects, and addition as its tensor product.  The tensor product of two morphisms is given by $\beta$ as in \ref{J_aop}, and it is simple to check that this is a strict monoidal structure.  The commutor $\sigma_{m,n}$ is $s_{1, m+n}s_{1,m}s_{m+1,m+n}$.  Using the relations in $J_{n}$, it is clear that $\sigma_{m,n}\sigma_{n,m}$ is the identity, so we only have one more axiom to verify in order to give a coboundary structure.  By definition, this axiom is equivalent to the equation
\[
\sigma_{m, p+n}\cdot \beta(e_{m}, \sigma_{n,p}) = \sigma_{n+m,p}\cdot \beta(\sigma_{m,n},e_{p})
\]
holding for all $m,n,p$.  Each side has six terms when written out using the definitions of $\sigma$ and $\beta$, two terms on each side cancel using $s_{p,q}^{2} = e$ and the disjointness relation, and the other four terms match after using the disjointness relation.  This establishes the coboundary structure on $B\mb{J}$; note that $\sigma_{1,1} = s_{1,2}$, the nontrivial element of $J(2)$.

Every strict coboundary functor $F:B\mb{J} \to X$ determines an object of $X$ by evaluation at $1$.  Conversely, given an object $x$ of a strict coboundary category $X$, we have an action of $J_{n}$ on $X(x^{n},x^{n})$ by Theorem 7 of \cite{hk-cobound} and therefore a strict monoidal functor $\overline{x}:B\mb{J} \to X$ with $\overline{x}(1) = x$.  By construction, this strict monoidal functor is in fact a strict coboundary functor since it sends the commutor $\sigma_{1,1}$ in $B\mb{J}$ to $\sigma_{x,x}$ in $X$.  In fact, the calculations in \cite{hk-cobound} leading up to Theorem 7 show that every element of $J_{n}$ is given as an operadic composition of $\sigma$'s, so requiring $\overline{x}$ to be a strict coboundary functor with $\overline{x}(1) = x$ determines the rest of the functor uniquely.  This establishes the bijection between strict coboundary functors $F:B\mb{J} \to X$ and objects of $X$ which proves that $B\mb{J}$ is the free strict coboundary category on one object.
\end{proof}

\begin{cor}
The 2-monad $C$ for coboundary categories corresponds, using  \cref{thm:club=operad}, to the action operad $\mb{J}$.
\end{cor}

\section{Profunctors and multicategories}
In this section we generalize from operads to multicategories (or colored operads).  The notions of plain and symmetric multicategories are standard \cite{bd_hda3}, but in fact there is a corresponding notion of $\mb{\Lambda}$-multicategory for any action operad $\mb{\Lambda}$.  We will give the basic definition and then show that it arises abstractly from a lifting of $E\Lambda$ as a 2-monad  on $\mb{Cat}$ to a pseudomonad on $\mb{Prof}$, the bicategory of categories, profunctors, and transformations.  A quick treatment of similar material but restricted to the symmetric case can be found in \cite{garner_poly}.

\begin{Defi}\label{lambda_multicat}
Let $\mb{\Lambda}$ be an action operad.  A \emph{$\mb{\Lambda}$-multicategory} $M$ consists of the following data:
\begin{itemize}
\item a set of objects $M_{0}$;
\item for any finite list $x_{1}, \ldots, x_{n}$ of objects and any object $y$, a set
\[
M(x_{1}, \ldots, x_{n}; y)
\]
of multi-arrows (or just arrows) from $x_{1}, \ldots, x_{n}$ to $y$;
\item for each $\alpha \in \Lambda(n)$, an isomorphism
\[
-\cdot \alpha : M(x_{1}, \ldots, x_{n}; y) \rightarrow M(x_{\pi(g)(1)}, \ldots, x_{\pi(g)(n)}; y);
\]
\item for each object $x$, an arrow $\id_{x} \in M(x;x)$; and
\item a composition function
\[
M(y_{1}, \ldots, y_{k}; z) \times M(x_{11}, \ldots, x_{1,n_{1}}; y_{1}) \times \cdots \times M(x_{k1}, \ldots, x_{k,n_{k}}; y_{k}) \rightarrow M(\underline{x}; z)
\]
where $\underline{x} = x_{11}, \ldots, x_{1,n_{1}}, x_{21}, \ldots, x_{k,n_{k}}$, and which we write as
\[
(g; f_{1}, \ldots, f_{n}) \mapsto g(f_{1}, \ldots, f_{n}).
\]
\end{itemize}
These data are subject to the following axioms.
\begin{enumerate}
\item $\id$ is a two-sided unit:
\[
\begin{array}{rcl}
\id (f) & = & f, \\
f(\id, \ldots, \id) & = & f.
\end{array}
\]
\item Composition is associative:
\[
f\Big( g_{1}(h_{11}, \ldots, h_{1m_{1}}), \ldots, g_{n}(h_{n1}, \ldots, h_{nm_{n}}) \Big) = f(g_{1}, \ldots, g_{n})(h_{11}, \ldots, h_{nm_{n}}).
\]
\item Composition respects the group actions:
\[
\begin{array}{rcl}
f(g_{1} \cdot \alpha_{1}, \ldots, g_{n} \cdot \alpha_{n}) & = & f(g_{1}, \ldots, g_{n}) \cdot \mu^{\Lambda}(e; \alpha_{1}, \ldots, \alpha_{n}), \\
f\cdot \alpha (g_{1}, \ldots, g_{n}) & = & f(g_{\pi^{-1}(\alpha)(1)}, \ldots, g_{\pi^{-1}(\alpha)(n)}) \cdot \mu^{\Lambda}(\alpha; e_{1}, \ldots, e_{n}).
\end{array}
\]
\end{enumerate}
\end{Defi}

\begin{Defi}
Let $M, N$ be $\mb{\Lambda}$-multicategories.  A \emph{$\mb{\Lambda}$-multifunctor} $F$ consists of the following data:
\begin{itemize}
\item a function $F_{0}:M_{0} \to N_{0}$ on sets of objects and
\item functions $F:M(x_1, \ldots, x_n; y) \to N(F_{0}(x_1), \ldots, F_{0}(x_n); F_{0}(y))$ which are $\Lambda(n)$-equivariant in that $F(f \cdot \alpha) = F(f) \cdot \alpha$.
\end{itemize}
These data are subject to the following axioms.
\begin{enumerate}
\item $F$ preserves identites: $F(\id_x) = \id_{F_{0}(x)}$.
\item $F$ preserves composition: $F\Big( f(g_1, \ldots, g_n) \Big) = F(f) \Big( F(g_1), \ldots, F(g_n) \Big).$
\end{enumerate}
\end{Defi}

Recall that the bicategory $\mb{Prof}$ has objects categories, 1-cells $F:X \srarrow Y$ profunctors from $X$ to $Y$ or equivalently functors
\[
F:Y^{\textrm{op}} \times X \rightarrow \mb{Sets},
\]
and 2-cells transformations $F \Rightarrow G$.  Composition of profunctors is given by the coend formula
\[
G \circ F (z,x) = \int^{y \in Y} G(z,y) \times F(y,x)
\]
and hence is only unital and associative up to coherent isomorphism.  There is an embedding pseudofunctor $(-)^{+}: \mb{Cat} \hookrightarrow \mb{Prof}$ which is the identity on objects and sends a functor $F:X \to Y$ to the profunctor $F^{+}$ defined by $F^{+}(y,x) = Y(y,Fx)$.

\begin{thm}
The 2-monad $E\Lambda$ on the 2-category $\mb{Cat}$ lifts to a pseudomonad $\widetilde{E\Lambda}$ on the bicategory $\mb{Prof}$.
\end{thm}
\begin{proof}
On objects, we have $\widetilde{E\Lambda}(X) = E\Lambda(X)$.  Let $F: X \srarrow Y$ be a profunctor given by the functor $F:Y^{\textrm{op}} \times X \rightarrow \mb{Sets}$.  We define $\widetilde{E\Lambda}F$ to be the functor
\[
( E\Lambda(Y) )^{\textrm{op}} \times E\Lambda(X) \rightarrow \mb{Sets}
\]
which is defined by the formulas
\[
\widetilde{\Lambda}F \big( [e; x_1, \ldots, x_n], [e; y_1, \ldots, y_m] \big) = \left\{
\begin{array}{lr}
\varnothing & \textrm{if $n \neq m$}, \\
\coprod_{g \in \Lambda(n)} \prod_{i=1}^{n} F(y_i, x_{\pi(g)(i)}) & \textrm{if $n = m$.}
\end{array}
\right.
\]
For a functor $G:X \to Y$, it is easy to check that
\[
\widetilde{E\Lambda}(G^{+}) = \big( E\Lambda G \big)^{+}
\]
using \ref{hom-set-lemma}.  The same formulas define the action of  $\widetilde{E\Lambda}$ on 2-cells as well.  The multiplication and unit of $\widetilde{E\Lambda}$ are just $\mu^{+}$ and $\eta^{+}$, where $\mu, \eta$ are the multiplication and unit, respectively, of $E\Lambda$.  The remainder of the pseudomonad data comes from the pseudofunctoriality of $(-)^{+}$, and the axioms follow from the 2-monad axioms for $E\Lambda$ and the pseudofunctor axioms for $(-)^{+}$.
\end{proof}

\begin{rem}
Since $\mb{Prof}$ is essentially the Kleisli bicategory for the free cocompletion pseudomonad, this lift corresponds to a pseudo-distributive law between $E\Lambda$ and the free cocompletion pseudomonad, but we do not pursue this perspective here.
\end{rem}

Given a bicategory $B$ and a pseudomonad $T$ on $B$, we can form the Kleisli bicategory of $T$, $\mb{Kl}_{T}$.  It has the same objects as $B$, but a 1-cell from $a$ to $b$ in  $\mb{Kl}_{T}$ is a 1-cell $f:a \rightarrow Tb$ in $B$.  In the case $B = \mb{Prof}, T = \widetilde{E\Lambda}$, the objects of $\mb{Kl}_{T}$ are categories, the 1-cells $X \srarrow Y$ are profunctors from $X$ to $E\Lambda Y$, or alternatively a functor $(E\Lambda Y)^{op} \times X \to \mb{Sets}$, and the 2-cells are natural transformation between such.

We now recall some standard definitions \cite{ben-bicats}.

\begin{Defi}
Let $B$ be a bicategory.  A \emph{monad} $(x,t,\mu,\eta)$ in $B$ consists of the following data:
\begin{itemize}
\item an object $x$,
\item a 1-cell $t: x \to x$,
\item a 2-cell $\mu:t^{2} \Rightarrow t$, and
\item a 2-cell $\eta: \id_x \Rightarrow t$.
\end{itemize}
These data are subject to the following axioms.
\[
\xy
(0,0)*+{(t \circ t) \circ t} ="1";
(25,0)*+{t \circ (t \circ t)} ="2";
(40,-12)*+{t \circ t} ="3";
(0,-24)*+{t \circ t} ="4";
(40,-24)*+{t} ="5";
{\ar^{\cong} "1";"2" };
{\ar^{t * \mu} "2";"3" };
{\ar^{\mu} "3";"5" };
{\ar_{\mu * t} "1";"4" };
{\ar_{\mu} "4";"5" };
(60,0)*+{\id_{x} \circ t} ="11";
(90,0)*+{t \circ t} ="12";
(90,-10)*+{t} ="13";
{\ar^{\eta * t} "11";"12" };
{\ar^{\mu} "12";"13" };
{\ar_{\cong} "11";"13" };
(60,-16)*+{t \circ \id_{x}} ="11";
(90,-16)*+{t \circ t} ="12";
(90,-26)*+{t} ="13";
{\ar^{t * \eta} "11";"12" };
{\ar^{\mu} "12";"13" };
{\ar_{\cong} "11";"13" };
\endxy
\]
\end{Defi}

\begin{Defi}
Let $(x,t,\mu,\eta), (x',t',\mu',\eta')$ be monads in $B$.  An \emph{oplax monad map} $(F, \alpha)$ from $t$ to $t'$ consists of the following data:
\begin{itemize}
\item a 1-cell $F:x \to x'$ and
\item a 2-cell $\alpha: F \circ t \Rightarrow t' \circ F$.
\end{itemize}
These data are subject to the following axioms, in which we suppress the constraints of the bicategory $B$.
\[
\xy
(0,0)*+{Ft^{2}} ="1";
(25,0)*+{t'Ft} ="2";
(40,-12)*+{t'^{2} F} ="3";
(0,-24)*+{Ft} ="4";
(40,-24)*+{t'F} ="5";
{\ar^{\alpha * t} "1";"2" };
{\ar^{t' * \alpha} "2";"3" };
{\ar^{\mu' * F} "3";"5" };
{\ar_{F * \mu} "1";"4" };
{\ar_{\alpha} "4";"5" };
(60,0)*+{F} ="11";
(90,0)*+{Ft} ="12";
(90,-10)*+{t'F} ="13";
{\ar^{F*\eta} "11";"12" };
{\ar^{\alpha} "12";"13" };
{\ar_{\eta'*F} "11";"13" };
\endxy
\]
\end{Defi}

\begin{Defi}
Let $(F,\alpha), (F', \alpha')$ be oplax monad maps from $t$ to $t'$.  A \emph{transformation of monad maps} $\Gamma: (F, \alpha) \Rightarrow (F', \alpha')$ is a 2-cell $\Gamma: F \Rightarrow F'$ such that
\[
\xy
(0,0)*+{Ft} ="1";
(40,0)*+{t'F} ="2";
(40,-12)*+{t'F'} ="3";
(0,-12)*+{F't} ="4";
{\ar^{\alpha } "1";"2" };
{\ar^{t' * \Gamma} "2";"3" };
{\ar_{\Gamma * t} "1";"4" };
{\ar_{\alpha'} "4";"3" };
\endxy
\]
commutes.
\end{Defi}

It is simple to check that monads, oplax monad maps, and transformations of monad maps form a bicategory.

\begin{thm}
There is a biequivalence between the category $\mb{\Lambda}\mbox{-}\mb{Multicat}$ of
\begin{itemize}
\item $\mb{\Lambda}$-multicategories and
\item $\mb{\Lambda}$-multifunctors, and
\end{itemize}
 the bicategory $\mb{Mnd}_{d}(\mb{Kl}_{\widetilde{E\Lambda}})$ of
\begin{itemize}
\item monads on sets (viewed as discrete categories) in $\mb{Kl}_{\widetilde{E\Lambda}}$,
\item oplax monad maps $(F, \alpha)$ between them which are isomorphic to one of the form $(f^{+}, \alpha)$ for $f:S \to T$ for some function of the underlying sets, and
\item transformations of monad maps.
\end{itemize}
Under this biequivalence, the category of $\mb{\Lambda}$-operads is equivalent to the bicategory of monads on the terminal set in $\mb{Kl}_{\widetilde{E\Lambda}}$.
\end{thm}
\begin{proof}
First, we note that $\mb{Mnd}_{d}(\mb{Kl}_{\widetilde{E\Lambda}})$ is a locally essentially discrete bicategory, by which we mean the hom-categories are all equivalent to discrete categories.  We will show there is a unique isomorphism or no 2-cell at all between oplax monad maps of the form $(f^{+}, \alpha)$, from which the claim follows in general.  A 2-cell between such has as its data a natural transformation $\gamma: f^{+} \Rightarrow g^{+}$ which has components
\[
\gamma_{[e; t_1, \ldots, t_n], s}:f^{+}([e; t_1, \ldots, t_n], s) \to g^{+}([e; t_1, \ldots, t_n], s).
\]
Both of these sets are empty unless $n=1$, and then the source is nonempty when $f(s) = t$ and the target is nonempty when $g(s)=t$; when nonempty, both of these sets are singletons.  If both are nonempty for some $s$, then the functions $f,g$ agree on $s$.  Assume the target is nonempty for some $([e;t], s)$ but that the source is empty, in other words that $g(s)=t$ but $f(s) \neq t$.  Then consider $\gamma_{[e;f(s)], s}$.  Its source is $f^{+}([e;f(s)], s)$ which is nonempty by construction, but its target is $g^{+}([e;f(s)], s)$.  We know that $g(s) = t \neq f(s)$, so $g^{+}([e;f(s)], s)$ must be empty, giving a map from a nonempty set to an empty one, a contradiction.  Thus there is a at most one 2-cell from an oplax monad map $(f^{+}, \alpha)$ to another $(g^{+}, \beta)$, such a map can only exist if $f = g$, and if it does exist then it is invertible.  Thus the hom-categories of $\mb{Mnd}_{d}(\mb{Kl}_{\widetilde{E\Lambda}})$ are essentially discrete, and this bicategory is equivalent to a category.

We begin by describing an object of $\mb{Mnd}_{d}(\mb{Kl}_{\widetilde{E\Lambda}})$ which is a monad in $\mb{Kl}_{\widetilde{E\Lambda}}$ whose underlying category is a set $S$.  A 1-cell $M:S \srarrow S$ is then a functor $(E\Lambda S)^{op} \times S \to \mb{Sets}$ which amounts to sets $M(s_1, \ldots, s_n; s)$ for $s_1, \ldots, s_n, s \in S$ together with a right action of $\Lambda(n)$ as in \ref{lambda_multicat}.  A 2-cell $1_{S} \Rightarrow M$ consists of a $\Lambda(1)$-equivariant function $\Lambda(1) \to M(s;s)$ for each $s \in S$, in other words an element $\id_{s} \in M(s;s)$.  A 2-cell $M \circ M \Rightarrow M$ then consists of a multicategorical composition function, as in \ref{lambda_multicat}, with appropriate equivariance built in by the coend used for composition of profunctors.  Associativity and unit conditions are then seen to be the same as for $\mb{\Lambda}$-multicategories.

By definition, an oplax monad map $(f^{+}, \alpha): (S,M) \to (S', M')$ consists of a function $f:S \to S'$ and a transformation $\alpha: M \circ f^{+} \Rightarrow f^{+} \circ M'$ satisfying two axioms.  The transformation $\alpha$ amounts to giving $\Lambda(n)$-equivariant functions
\[
M(s_1, \ldots, s_n; s) \to M'\big(f(s_1), \ldots, f(s_n); f(s)\big),
\]
and the two axioms correspond to the unit and composition axioms for a $\mb{\Lambda}$-multifunctor.

These descriptions give the action on objects and morphisms of a pseudofunctor $\mb{\Lambda}\mbox{-}\mb{Multicat} \to \mb{Mnd}_{d}(\mb{Kl}_{\widetilde{E\Lambda}})$ with local contractibility providing the pseudofunctoriality constraints as well as showing that the axioms for a pseudofunctor hold.  It is also clear that this pseudofunctor is biessentially surjective and locally essentially surjective, so it is a biequivalence once again using local contractibility.

The final claim is then an immediate consequence of the definitions of $\mb{\Lambda}$-operad and $\mb{\Lambda}$-multicategory.
\end{proof}

\bibliographystyle{plain}
\bibliography{borel_ref}

\end{document}